\documentclass[reqno]{amsart}
\usepackage[american]{babel}
\usepackage{amsmath,amsfonts,amsthm,amssymb,latexsym,euscript}
\numberwithin{equation}{section}

\newtheorem{theorem}{Theorem}[section]
\newtheorem{lemma}{Lemma}[section]
\newtheorem{corollary}{Corollary}[section]
\theoremstyle{definition}
\newtheorem{definition}{Definition}[section]
\theoremstyle{remark}
\newtheorem{remark}{Remark}[section]

\DeclareMathOperator{\sgn}{sign}

\begin{document}

\title[Zakharov--Kuznetsov equation]
{An initial-boundary value problem in a strip for a two-dimensional equation of Zakharov--Kuznetsov type}

\author[A.V.~Faminskii]{Andrei~V.~Faminskii}


\address{Department of Mathematics, Peoples' Friendship University of Russia,
Miklukho--Maklai str. 6, Moscow, 117198, Russia}
\email{afaminskii@sci.pfu.edu.ru}

\date{}
\maketitle

{\scriptsize \centerline{Peoples' Friendship University of Russia, Moscow, Russia}}

\begin{abstract}
An initial-boundary value problem in a strip with homogeneous Dirichlet boundary conditions for two-dimensional generalized Zakharov--Kuznetsov equation is considered. In particular, dissipative and absorbing degenerate terms can be supplemented to the original Zakharov--Kuznetsov equation. Results on global existence, uniqueness and long-time decay of weak solutions are established.
\end{abstract}

\section{Introduction. Description of main results}\label{S1}
Two dimensional Zakharov--Kuznetsov equation (ZK)
$$
u_t+u_{xxx}+u_{xyy}+uu_x=0
$$
models propagation of ion-acoustic waves in magnetized plasma, \cite{ZK}. A rigorous derivation of the ZK model was recently performed in \cite{LLS}. Results on well-posedness of the initial value problem for this equation can be found in \cite{S, F89, F95, LP}. A theory of well-posedness of initial-boundary value problems is most developed for domains of a type $I\times \mathbb R$, where $I$ is an interval (bounded or unbounded) on the variable $x$, that is the variable $y$ varies in the whole line, \cite{F07, FB07, FB08, F08, ST, F12, DL}. On the contrary, there are only a few results for domains, where the variable $y$ varies in a bounded interval.

In \cite{LPS} an initial-boundary value problem in a strip $\mathbb R\times (0,2\pi)$ with periodic boundary conditions is considered and local well-posedness result is established in the spaces $H^s$ for $s>3/2$. An initial-boundary value problem in a half-strip $\mathbb R_+\times (0,L)$ with homogeneous Dirichlet boundary conditions is studied in \cite{LT, L} and global well-posedness in Sobolev spaces with exponential weights when $x\to +\infty$ is proved. Initial-boundary value problems in a strip $\mathbb R\times (0,L)$ with homogeneous boundary conditions of different types: Dirichlet, Neumann or periodic are considered in \cite{BF13} and results on global well-posedness in classes of weak solutions with power weights at $+\infty$ are established. Global well-posedness results for a bounded rectangle can be found in \cite{DL, STW}.

All global existence results for ZK equation are based on a conservation law in $L_2$
$$
\iint_\Omega u^2(t,x,y)\,dxdy=\text{const},
$$
where $\Omega$ is a domain of a type $\mathbb R\times I$, or its analogues for other types of domains. Note that in the situations where such a conservation law exists there is no decay of solutions in $L_2$-norms when $t\to +\infty$.

On the other hand it is found in \cite{LT} that ZK equation possesses certain internal dissipation which can provide decay of solutions in weighted $L_2$-norms. More precisely, it is shown in \cite{LT} that in a narrow half-strip $\mathbb R_+\times (0,L)$ under small initial data and homogeneous Dirichlet boundary conditions a solution to the corresponding initial-boundary value problem decays exponentially when $t\to +\infty$ in $L_2$ spaces with exponential weight at $+\infty$. In \cite{L} with the use of the next conservation law
$$
\iint_\Omega \left(u_x^2+u_y^2+\frac 13 u^3\right)\,dxdy=\text{const}
$$
similar decay result is proved in $H^1$-norm (also in exponentially-weighted spaces).

In the recent paper \cite{L1} the similar result is obtained for ZK-type equation with additional damping term $-u_{xx}$ in a strip $\mathbb R\times (0,L)$ with homogeneous Dirichlet boundary conditions. Moreover, it was noticed there that under the appropriate choice of weight functions restrictions on the width of the strip can be excluded.

In \cite{DL} exponential long-time decay in $L_2$-norm of small solutions is established for initial-boundary value problems in a bounded rectangle and in a vertical strip.
 
The present paper is devoted to an initial-boundary value problem in a layer $\Pi_T=(0,T)\times\Sigma$, where $\Sigma=\mathbb R\times (0,L)=\{(x,y): x\in \mathbb R, 0<y<L\}$ is a strip of a given width $L$ and $T>0$ is arbitrary, for an equation
\begin{equation}\label{1.1}
u_t+bu_x+u_{xxx}+u_{xyy}+uu_x-\left(a_1(x,y)u_x\right)_x-\left(a_2(x,y)u_y\right)_y+a_0(x,y)u=f(t,x,y)
\end{equation}
with an initial condition 
\begin{equation}\label{1.2}
u(0,x,y)=u_0(x,y),\qquad (x,y)\in\Sigma,
\end{equation}
and homogeneous Dirichlet boundary conditions 

\begin{equation}\label{1.3}
u(t,x,0)=u(t,x,L)=0,\qquad (t,x)\in(0,T)\times\mathbb R.
\end{equation}

Here $bu_x$ is a travel term ($b$ is a real constant). We always assume that
\begin{equation}\label {1.4}
a_1(x,y), a_2(x,y)\geq 0\qquad \forall (x,y)\in \Sigma,
\end{equation}
so the corresponding terms in (\ref{1.1}) mean the parabolic damping which, in particular, can degenerate or even be absent.

The main goal of this paper is to study relations between internal properties of ZK equation itself and artificial damping implemented by dissipation or absorption which provide existence and uniqueness of global weak solutions to the considered problem as well as their long-time decay in $L_2$-norms with different weights (or without them).

Some of the established results are valid for ZK equation itself ($a_0=a_1=a_2\equiv 0$) and in this situation partially coincide with corresponding ones from \cite{BF13} (in the part related to existence and uniqueness).

In all results we assume that
\begin{equation}\label{1.5}
a_j\in L_\infty(\Sigma), \qquad j=0,1,2 
\end{equation}
(sometimes we need more smoothness). Besides that, most of the results are established in four different situations: 1) the parabolic damping can be absent; 2) it is effective at both infinities, that is there exist $a>0$, $R>0$ such that
\begin{equation}\label{1.6}
a_1(x,y), a_2(x,y)\geq a\qquad \text{if}\quad |x|\geq R;
\end{equation}	
3) it is effective at $-\infty$, that is there exist $a>0$, $R>0$ such that
\begin{equation}\label{1.7}
a_1(x,y), a_2(x,y)\geq a\qquad \text{if}\quad x\leq -R;
\end{equation}	
4) it is effective at $+\infty$, that is there exist $a>0$, $R>0$ such that
\begin{equation}\label{1.8}
a_1(x,y), a_2(x,y)\geq a\qquad \text{if}\quad x\geq R.
\end{equation}

Introduce the following notation. For an integer $k\geq 0$ let
$$
|D^k\varphi|=\Bigl(\sum_{k_1+k_2=k}(\partial^{k_1}_x\partial_y^{k_2}\varphi)^2\Bigr)^{1/2}, \qquad
|D\varphi|=|D^1\varphi|.
$$
Let $L_p=L_p(\Sigma)$, $W_p^k=W_p^k(\Sigma)$, $H^k=H^k(\Sigma)$, $x_+=\max(x,0)$, $\mathbb R_+=(0,+\infty)$, $\mathbb R_-=(-\infty,0)$, $\Sigma_{\pm}=\mathbb R_\pm\times (0,L)$, $\Pi^\pm_T=(0,T)\times\Sigma_{\pm}$, $L_{p,\pm}=L_p(\Sigma_\pm)$, $H^k_\pm=H^k(\Sigma_\pm)$.

For a measurable non-negative on $\mathbb R$ function $\psi(x)\not\equiv \text{const}$ let
$$
L_2^{\psi(x)} =\{\varphi(x,y): \varphi\psi^{1/2}(x)\in L_2\}
$$
with a natural norm. In particular important cases we use the special notation
$$
L_2^\alpha=L_2^{(1+x_+)^{2\alpha}}\quad \forall\ \alpha\in\mathbb R, \qquad L_2^{\alpha,exp}=L_2^{1+e^{2\alpha x}}\quad \forall\  \alpha>0.
$$
Let for an integer $k\geq 0$
$$
H^{k,\psi(x)}=\{\varphi: |D^j\varphi|\in L_2^{\psi(x)}, \  j=0,\dots,k\}
$$
with a natural norm,
$$
H^{k,\alpha}=H^{k,(1+x_+)^{2\alpha}}\quad \forall\ \alpha\in\mathbb R,\qquad H^{k,\alpha,exp}=H^{k,1+e^{2\alpha x}} \quad \forall\  \alpha>0.
$$
Restrictions of these spaces on $\Sigma_\pm$ are denoted by lower indices "+" and "-" respectively: $L_{2,+}^{\psi(x)}$,
$L_{2,-}^{\psi(x)}$, $H_+^{k,\psi(x)}$, $H_-^{k,\psi(x)}$ etc.

We say that $\psi(x)$ is an admissible weight function if $\psi$ is an infinitely smooth positive on $\mathbb R$ function such that $|\psi^{(j)}(x)|\leq c(j)\psi(x)$ for each natural $j$ and all $x\in\mathbb R$. Note that such a function has not more than exponential growth and not more than exponential decrease at $\pm\infty$. It is shown in \cite{F12} that $\psi^s(x)$ for any $s\in\mathbb R$ is also an admissible weight function.

As an important example of such functions introduce for $\alpha\geq 0$ special infinitely smooth functions $\rho_\alpha(x)$ in a following way: $\rho_\alpha(x)>1$,  $0<\rho'_\alpha(x)\leq c(\alpha)\rho_\alpha(x)$, 
$|\rho_\alpha^{(j)}(x)|\leq c(\alpha,j)\rho'_\alpha(x)$ for each natural $j\geq 2$ and all $x\in\mathbb R$, $\rho''_\alpha(x)>0$ for $x\leq-1$, $\rho_0(x)<2$, $\rho''_0(x)<0$ for $x\geq 1$, $\rho_\alpha(x)=(1+x)^{2\alpha}$ for $\alpha>0$ and $x\geq1$. It is easy to see that such functions exist and, moreover, for $\alpha\geq 0$
$$
L_2^{\rho_\alpha(x)}=L_2^\alpha, \qquad H^{k,\rho_\alpha(x)}=H^{k,\alpha}.
$$
Note that both $\rho_\alpha$ and $\rho'_\alpha$ are admissible weight functions.

We construct solutions to the considered problem in spaces $X^{k,\psi(x)}(\Pi_T)$, $k=0 \mbox{ or }1$, for admissible non-decreasing weight functions $\psi(x)\geq 1\ \forall x\in\mathbb R$, consisting of functions $u(t,x,y)$ such that 
\begin{equation}\label{1.9}
u\in C_w([0,T]; H^{k,\psi(x)}), \qquad 
|D^{k+1}u|\in L_2(0,T;L_2^{\psi'(x)})
\end{equation}
(the symbol $C_w$ denotes the space of  weakly continuous mappings),
\begin{equation}\label{1.10}
\lambda(|D^{k+1} u|;T) =
\sup_{x_0\in\mathbb R}\int_0^T\! \int_{x_0}^{x_0+1}\! \int_0^L |D^{k+1}u|^2\,dydxdt<\infty
\end{equation}
(let $X^{\psi(x)}(\Pi_T)=X^{0,\psi(x)}(\Pi_T)$). Restrictions of these spaces on $\Pi_T^\pm$ are denoted by
 $X^{k,\psi(x)}(\Pi_T^\pm)$ respectively.

In particular important cases we use the special notation
$$
X^{k,\alpha}(\Pi_T)=X^{k,\rho_\alpha(x)}(\Pi_T),\quad X^\alpha(\Pi_T)=X^{0,\alpha}(\Pi_T)
$$
and for $\alpha>0$
$$
X^{k,\alpha,exp}(\Pi_T)=X^{k,1+e^{2\alpha x}}(\Pi_T) 
$$
(with similar notation for restrictions on $\Pi_T^\pm$).
It is easy to see that $X^{k,0}(\Pi_T)$ coincides with a space of functions $u\in C_w([0,T]; H^k)$ for which \eqref{1.10} holds, $X^{k,\alpha}(\Pi_T)$, $\alpha>0$, --- with a space of functions $u\in C_w([0,T]; H^{k,\alpha})$  for which \eqref{1.10} holds and, in addition, $|D^{k+1} u|\in L_2(0,T;L_{2,+}^{\alpha-1/2})$, $X^{k,\alpha,exp}(\Pi_T)$ --- with a space of functions $u\in C_w([0,T]; H^{k,\alpha,exp})$  for which \eqref{1.10} holds and, in addition, 
$|D^{k+1} u|\in L_2(0,T;L_{2,+}^{\alpha,exp})$.

Now we can formulate results of the paper concerning existence and uniqueness of weak solutions.

\begin{theorem}\label{T1.1}
Let assumptions \eqref{1.4} and \eqref{1.5} be satisfied. Assume also that $u_0\in L_2^{\psi(x)}$, 
$f\in L_1(0,T; L_2^{\psi(x)})$ for certain $T>0$ and an admissible weight function $\psi(x)\geq 1\ \forall x\in\mathbb R$ such that $\psi'(x)$ is also an admissible weight function. Then there exists a weak solution to problem \eqref{1.1}--\eqref{1.3} $u \in X^{\psi(x)}(\Pi_T)$. If, in addition,

\noindent 1) assumption \eqref{1.6} holds then this solution $u\in C_w([0,T];L_2^{\psi(x)})\cap L_2(0,T;H^{1,\psi(x)})$ and is unique in this space;

\noindent 2) assumption \eqref{1.7} holds then $u\in C_w([0,T];L_2^{\psi(x)})\cap L_2(0,T;H_+^{1,\psi'(x)})\cap L_2(0,T;H_-^1)$ and is unique in this space if $\psi(x)\geq \rho_1(x) \ \forall x\in\mathbb R$, $\psi'(x)\geq \rho_{1/2}(x)\ \forall x\geq 0$;

\noindent 3) assumption \eqref{1.8} holds then $u\in L_2(0,T;H_+^{1,\psi(x)})$.
\end{theorem}

\begin{theorem}\label{T1.2}
Let assumption \eqref{1.4} be satisfied, $a_1,a_2\in W_\infty^2$, $a_0\in W_\infty^1$, $a_{2\,y}(x,0)\leq 0$, 
$a_{2\,y}(x,L)\geq 0$ $\forall x\in\mathbb R$. Assume also that $u_0\in H^{1,\psi(x)}$, $f\in L_1(0,T; H^{1,\psi(x)})$ for certain $T>0$ and an admissible weight function $\psi(x)\geq 1\ \forall x\in\mathbb R$ such that $\psi'(x)$ is also an admissible weight function, $u_0|_{y=0}=u_0|_{y=L}= 0$, $f|_{y=0}=f|_{y=L}= 0$ $\forall x\in\mathbb R$ and $\forall t\in (0,T)$. Then there exists a weak solution to problem \eqref{1.1}--\eqref{1.3} $u\in X^{1,\psi(x)}(\Pi_T)$ and it is unique in this space if $\psi(x)\geq \rho_{1/2}(x)$ $\forall x\in \mathbb R$.  If, in addition,

\noindent 1) assumption \eqref{1.6} holds then this solution $u\in C_w([0,T];H^{1,\psi(x)})\cap L_2(0,T;H^{2,\psi(x)})$;

\noindent 2) assumption \eqref{1.7} holds then $u\in C_w([0,T];H^{1,\psi(x)})\cap L_2(0,T;H_+^{2,\psi'(x)})\cap L_2(0,T;H_-^2)$;

\noindent 3) assumption \eqref{1.8} holds then $u\in X^{1,\psi(x)}(\Pi_T)\cap L_2(0,T;H_+^{2,\psi(x)})$ and is unique in this space.
\end{theorem}

\begin{remark}\label{R1.1}
The weight functions $\psi(x)\equiv\rho_\alpha(x)$, $\alpha\geq 0$, and $\psi(x)\equiv 1+e^{2\alpha x}$, $\alpha>0$, satisfy the hypotheses of Theorems~\ref{T1.1} and~\ref{T1.2}. In particular, if \eqref{1.4}--\eqref{1.6} hold, $u_0\in L_2$, $f\in L_1(0,T;L_2)$, there exists a unique solution in the space $C_w([0,T];L_2)\cap L_2(0,T;H^1)$. If \eqref{1.4}, \eqref{1.5} and \eqref{1.7} hold, $u_0\in L_2^{\alpha,exp}$, $f\in L_1(0,T;L_2^{\alpha,exp})$ for $\alpha>0$, there exists a unique solution in the space $C_w([0,T];L_2^{\alpha,exp})\cap L_2(0,T;H^{1,\alpha,exp})$.
\end{remark}

Next, pass to the decay results which can be considered as corollaries of Theorems~\ref{T1.1} and~\ref{T1.2}. Here we always assume that $f\equiv 0$. Then it is easy to see that one can construct solutions lying in the same spaces as in Theorems~\ref{T1.1} and~\ref{T1.2} for any $T>0$ even if they do not belong to the classes of uniqueness. 

\begin{corollary}\label{C1.1}
Let assumptions \eqref{1.4} and \eqref{1.5} be satisfied, $u_0\in L_2$, $f\equiv 0$. Assume also that
\begin{equation}\label{1.11}
a_2(x,y)\geq \beta_2(x)\geq 0, \quad a_0(x,y)\geq \beta_0(x)\qquad \forall (x,y)\in\Sigma
\end{equation}
for certain measurable functions $\beta_2$, $\beta_0$ and, moreover,
\begin{equation}\label{1.12}
\frac{\pi^2\beta_2(x)}{L^2}+\beta_0(x)\geq \beta=\text{const}>0\quad \forall x\in\mathbb R.
\end{equation}
Then there exists a weak solution to problem \eqref{1.1}--\eqref{1.3} $u\in X^0(\Pi_T)$ $\forall T>0$ such that
\begin{equation}\label{1.13}
\|u(t,\cdot,\cdot)\|_{L_2}\leq e^{-\beta t}\|u_0\|_{L_2}\qquad \forall t\geq 0.
\end{equation}
\end{corollary}

\begin{remark}\label{R1.2}
If additional assumptions on the data provide uniqueness according to Theorems~\ref{T1.1} and~\ref{T1.2} then, of course, any solution from the class of uniqueness possess property \eqref{1.13} (under \eqref{1.11}, \eqref{1.12}).   Similar remark is applicable also in the following results.
\end{remark}

Inequalities \eqref{1.11} and \eqref{1.12} mean that either dissipation or absorption must be effective at every point to ensure exponential decay. It is interesting to compare this result with the one-dimensional case. Korteweg--de~Vries (KdV) equation itself
$$
u_t+u_{xxx}+uu_x=0
$$
as well as ZK equation possesses the conservation law in $L_2$, so without additional damping there is no decay of solutions to the initial value problem.

Consider Korteweg--de~Vries--Burgers equation
$$
u_t+u_{xxx}+uu_x-a_1u_{xx}=0,\qquad a_1=\text{const}>0.
$$
It is proved in \cite{ABS} that for $u_0\in L_2(\mathbb R)\cap L_1(\mathbb R)$ a corresponding solution to the initial value problem satisfies an inequality
$$
\|u(t,\cdot)\|_{L_2(\mathbb R)}\leq c(1+t)^{-1/4}\qquad \forall t\geq 0
$$
and this result is sharp, so here dissipation provides only power decay.

Of course, if one considers KdV type equation with absorption damping on the whole real line
$$
u_t+u_{xxx}+uu_x+a_0u=0, \qquad a_0=\text{const}>0,
$$
then it is easy to see that a corresponding solution to the initial value problem decays exponentially:
$$
\|u(t,\cdot)\|_{L_2(\mathbb R)}\leq e^{-a_0t}\|u_0\|_{L_2(\mathbb R)} \qquad \forall t\geq 0.
$$
It is shown in \cite{CCFN} that exponential decay remains even in the case of a localized absorption, that is for the initial value problem for an equation
$$
u_t+u_{xxx}+uu_x+a_0(x)u=0
$$
if $a_0(x)\geq 0$ $\forall x\in\mathbb R$, $a_0(x)\geq\beta>0$ for $|x|\geq R$, then
$$
\|u(t,\cdot)\|_{L_2(\mathbb R)}\leq c e^{-c_0t}\qquad \forall t\geq 0,
$$
where positive constants $c$ and $c_0$ are uniform for initial data $u_0$ from any bounded set in $L_2(\mathbb R)$.

Similar result for equation \eqref{1.1} if dissipation is effective at both infinities is obtained in this paper for small solutions and its proof is based on ideas from \cite{LT,L1}.

\begin{corollary}\label{C1.2}
Let assumptions \eqref{1.4}--\eqref{1.6} be satisfied and, in addition, $a_0(x,y)\geq 0$ $\forall (x,y)\in\Sigma$. Assume also that $u_0\in L_2$, $f\equiv 0$. Then there there exist $\epsilon_0>0$ and $\beta>0$ such that if $\|u_0\|_{L_2}\leq\epsilon_0$ the corresponding unique weak solution $u(t,x,y)$ to problem \eqref{1.1}--\eqref{1.3} from the space $C_w([0,T];L_2)\cap L_2(0,T;H^1)$ $\forall T>0$ satisfies an inequality
\begin{equation}\label{1.14}
\|u(t,\cdot,\cdot)\|_{L_2}\leq \sqrt{2}e^{-\beta t}\|u_0\|_{L_2}\qquad \forall t\geq 0.
\end{equation}
\end{corollary}

In the spaces with exponential weights at $+\infty$ a similar result is established without any additional damping but with certain restrictions on the width of the strip in the case $b>0$.

\begin{corollary}\label{C1.3}
Let assumptions \eqref{1.4}, \eqref{1.5} be satisfied and, in addition, $a_0(x,y)\geq 0$ $\forall (x,y)\in\Sigma$. Then let $L_0=+\infty$ if $b\leq 0$ and if $b>0$ there exists $L_0>0$ such that in both cases for any $L\in (0,L_0)$ there exist $\alpha_0>0$, $\epsilon_0>0$ and $\beta>0$ such that if $u_0\in L_2^{\alpha,exp}$ for $\alpha\in (0,\alpha_0]$, $\|u_0\|_{L_2}\leq\epsilon_0$, $f\equiv 0$ there exists a weak solution $u(t,x,y$) to problem \eqref{1.1}--\eqref{1.3} from the space $X^{\alpha,exp}(\Pi_T)$ $\forall T>0$ satisfying an inequality
\begin{equation}\label{1.15}
\|e^{\alpha x}u(t,\cdot,\cdot)\|_{L_2}\leq e^{-\alpha\beta t}\|e^{\alpha x}u_0\|_{L_2}\qquad \forall t\geq 0.
\end{equation}
\end{corollary}

If parabolic damping is effective at $-\infty$ this result can be improved.

\begin{corollary}\label{C1.4}
Let assumptions \eqref{1.4}, \eqref{1.5}, \eqref{1.7} be satisfied and, in addition, $a_0(x,y)\geq 0$ $\forall (x,y)\in\Sigma$. Then let $L_0=+\infty$ if $b\leq 0$ and if $b>0$ there exists $L_0>0$ such that in both cases for any $L\in (0,L_0)$ there exist $\alpha_0>0$, $\epsilon_0>0$ and $\beta>0$ such that if $u_0\in L_2^{\alpha,exp}$ for  $\alpha\in (0,\alpha_0]$, $\|u_0\|_{L_2}\leq\epsilon_0$, $f\equiv 0$ the corresponding unique weak solution $u(t,x,y)$ to problem \eqref{1.1}--\eqref{1.3} from the space $C_w([0,T];L_2^{\alpha,exp})\cap L_2(0,T;H^{1,\alpha,exp})$ $\forall T>0$ satisfies an inequality
\begin{equation}\label{1.16}
\|(1+e^{2\alpha x})^{1/2}u(t,\cdot,\cdot)\|_{L_2}\leq e^{-\alpha\beta t}\|(1+e^{2\alpha x})^{1/2}u_0\|_{L_2}
\qquad \forall t\geq 0.
\end{equation}
\end{corollary}

In order to present a result when parabolic damping is effective at $+\infty$, introduce the following auxiliary functions.  
For each $\alpha\geq 0$ define an infinitely smooth increasing on $\mathbb R$ function $\kappa_{\alpha}(x)$ as follows: 
$\kappa_{\alpha}(x)=e^{2\alpha x}$ when $x\leq -1$, $\kappa_{\alpha}(x)=(1+x)^{2\alpha}$ for $\alpha>0$ and $\kappa_0(x)=2-(1+x)^{-1/2}$ when $x\geq 0$, $\kappa'_{\alpha}(x)>0$ when $x\in (-1,0)$. Note that both $\kappa_{\alpha}$ and $\kappa'_{\alpha}$ are admissible weight functions, and $\kappa'_{\alpha}(x)\leq c(\alpha)\kappa_{\alpha}(x)$ for all $x\in \mathbb R$.

\begin{corollary}\label{C1.5}
Let assumptions \eqref{1.4}, \eqref{1.5}, \eqref{1.8} be satisfied and, in addition, $a_0(x,y)\geq 0$ $\forall (x,y)\in\Sigma$. Then let $L_0=+\infty$ if $b\leq 0$ and if $b>0$ there exists $L_0>0$ such that in both cases for any $L\in (0,L_0)$ there exist $\alpha_0>0$, $\epsilon_0>0$ and $\beta>0$ such that if $u_0\in L_2$, $\|u_0\|_{L_2}\leq\epsilon_0$, $f\equiv 0$ there exists a weak solution $u(t,x,y$) to problem \eqref{1.1}--\eqref{1.3} satisfying $u\in X^0(\Pi_T)$, $|D u|\in L_2(\Pi_T^+)$ $\forall T>0$ and for any $\alpha\in (0,\alpha_0]$
\begin{equation}\label{1.17}
\|\kappa_0^{1/2}(\alpha x)u(t,\cdot,\cdot)\|_{L_2}\leq e^{-\alpha\beta t}\|\kappa_0^{1/2}(\alpha x)u_0\|_{L_2}\qquad \forall t\geq 0.
\end{equation}
\end{corollary}

Further let $\eta(x)$ denotes a cut-off function, namely, $\eta$ is an infinitely smooth non-decreasing on $\mathbb R$ function such that $\eta(x)=0$ when $x\leq 0$, $\eta(x)=1$ when $x\geq 1$, $\eta(x)+\eta(1-x)\equiv 1$.

We omit limits of integration in integrals over the whole strip $\Sigma$.

The following interpolating inequality generalizing the one from \cite{LSU} for weighted Sobolev spaces is crucial for the study.

\begin{lemma}\label{L1.1}
Let $\psi_1(x)$, $\psi_2(x)$ be two admissible weight functions such that $\psi_1(x)\leq c_0\psi_2(x)$ $\forall x\in \mathbb R$ for some constant $c_0>0$. Let $k$ be natural, $m\in [0,k)$ -- integer, $q\in [2,+\infty]$ if $k-m\geq 2$ and $q\in [2,+\infty)$ in other cases. For the case $q=+\infty$  assume also that $\displaystyle{\frac{\psi_2(x_1)}{\psi_1(x_1)}\leq c_0\frac{\psi_2(x_2)}{\psi_1(x_2)}}$ if $|x_1-x_2|\leq 1$. Then there exists a constant $c>0$ such that for every function $\varphi(x,y)$ satisfying $|D^k\varphi|\psi_1^{1/2}(x)\in L_2$, $\varphi\psi_2^{1/2}(x)\in L_2$, the following inequality holds
\begin{equation}\label{1.18}
\bigl\| |D^m\varphi|\psi_1^s(x)\psi_2^{1/2-s}(x)\bigr\|_{L_q} \leq c 
\bigl\| |D^k\varphi|\psi_1^{1/2}(x)\bigr\|^{2s}_{L_2}
\bigl\| \varphi\psi_2^{1/2}(x)\bigr\|^{1-2s}_{L_2} 
+c\bigl\|\varphi\psi_2^{1/2}(x)\bigr\|_{L_2},
\end{equation}
where $\displaystyle{s=\frac{m+1}{2k}-\frac{1}{kq}}$. If $\varphi\big|_{y=0}=\varphi\big|_{y=L}=0$ and either $k=1, m=0$ or $k=2, m=0, q<+\infty$ or $k=2, m=1, q=2$ then the constant $c$ in \eqref{1.18} is uniform with respect to $L$.
\end{lemma}

\begin{proof}
If one considers the whole plane $\mathbb R^2$ instead of the strip $\Sigma$ the given inequality is a special case for a more general interpolating inequality, estimated in \cite{F89} for an arbitrary number of variables. The proof in this case is similar. Therefore, we reproduce it here only for three aforementioned cases (either $k=1, m=0$ or $k=2, m=0, q<+\infty$ or $k=2, m=1, q=2$), in particular, to make it clear why the constant $c$ is independent on $L$ in these cases when $\varphi\big|_{y=0}=\varphi\big|_{y=L}=0$.

Without loss of generality assume that $\varphi$ is a smooth decaying at $\infty$ function. First following \cite{LSU} estimate one auxiliary inequality (which itself is also used later): for $p\in [1,2)$, $p^*=2p/(2-p)$ uniformly with respect to $L$
\begin{equation}\label{1.19}
\|\varphi\|_{L_{p^*}}\leq \frac{c(p)}{L}\bigl\| |D\varphi|+|\varphi|\bigr\|_{L_p}, \qquad
\|\varphi\|_{L_{p^*}}\leq c(p)\bigl\| |D\varphi|\bigr\|_{L_p} \quad \text{if}\ \ 
\varphi\big|_{y=0}=\varphi\big|_{y=L}=0.
\end{equation}
For $p=1$ (then $p^*=2$) this inequality follows from an inequality
$$
\iint \varphi^2\,dxdy \leq \int_0^L \sup\limits_{x\in\mathbb R} |\varphi|\,dy 
\int_{\mathbb R} \sup\limits_{y\in (0,L)} |\varphi|\,dx
$$
and obvious interpolating one-dimensional inequalities
\begin{gather*}
\sup\limits_{x\in\mathbb R} |\psi|\leq \int_{\mathbb R} |\psi'|\,dx, \qquad
\sup\limits_{y\in (0,L)} |\psi| \leq \frac{c}{L} \int_0^L \bigl(|\psi'|+|\psi|)\,dy, \\
\sup\limits_{y\in (0,L)} |\psi| \leq \int_0^L |\psi'|\,dy\quad \text{if}\ \ \psi(0)=0 \ \ \text{or}\ \ \psi(L)=0.
\end{gather*}
If $p\in (1,2)$ let $\widetilde\varphi\equiv |\varphi|^{p^*/2}\sgn{\varphi}$, then in the general case the first one of inequalities \eqref{1.19} for $p=1$ yields  that
\begin{multline*}
\|\varphi\|_{L_{p^*}}^{p^*/2}=\|\widetilde\varphi\|_{L_2}\leq 
\frac{c}{L} \bigl\| |D\widetilde\varphi|+|\widetilde\varphi|\bigr\|_{L_1} \leq
\frac{c(p)}{L} \bigl\| |\varphi|^{p^*/2-1}(|D\varphi|+|\varphi|)\bigr\|_{L_1} \\ \leq
\frac{c(p)}{L} \| |\varphi|^{p^*/2-1}\|_{L_{p/(p-1)}} \bigl\| |D\varphi|+|\varphi|\bigr\|_{L_p} =
\frac{c(p)}{L} \|\varphi\|_{L_{p^*}}^{p^*/2-1} \bigl\| |D\varphi|+|\varphi|\bigr\|_{L_p},
\end{multline*}
whence \eqref{1.19} in this case follows. If $\varphi|_{y=0}=\varphi|_{y=L}=0$ one has to repeat this argument with the use of the second one of inequalities \eqref{1.19} for $p=1$.

Now we can prove estimate \eqref{1.18} for $k=1$, $m=0$, $q>2$ (for $q=2$ it is obvious). In fact, in the general case choosing  $p\in (1,2)$ such that $q<p^*$ and applying first H\"older inequality, then the first one of inequalities \eqref{1.19} to the function $\Phi\equiv |\varphi|^{2/p}\psi_1^{1/2}\psi_2^{1/p^*}\sgn{\varphi}$ 
(note that $|D\Phi|\leq c(q)\bigl(|D\varphi|+|\varphi|\bigr)\psi_1^{1/2}|\varphi|^{2/p^*}\psi_2^{1/p^*}$) and finally again  H\"older inequality we derive that
\begin{multline*}
\|\varphi\psi_1^{1/2-1/q}\psi_2^{1/q}\|_{L_q} =
\bigl\| |\Phi|^{q-2} \bigl(|\varphi|\psi_2^{1/2}\bigr)^{q-2(q-2)/p}\bigr\|_{L_1}^{1/q}  \\ \leq
\|\Phi\|_{L_{p^*}}^{1-2/q}\|\varphi\psi_2^{1/2}\|_{L_2}^{1-2(q-2)/(pq)} \leq
\frac{c(q)}{L^{1-2/q}} \bigl\| |D\Phi|+|\Phi| \bigr\|_{L_p}^{1-2/q}
\|\varphi\psi_2^{1/2}\|_{L_2}^{1-2(q-2)/(pq)} \\ \leq
\frac{c_1(q)}{L^{1-2/q}} \bigl\| \bigl(|D\varphi|+|\varphi|\bigr)\psi_1^{1/2}
|\varphi|^{2/p^*}\psi_2^{1/p^*}\bigr\|_{L_p}^{1-2/q} \|\varphi\psi_2^{1/2}\|_{L_2}^{1-2(q-2)/(pq)}  \\ \leq
\frac{c_1(q)}{L^{1-2/q}} \bigl\| \bigl(|D\varphi|+|\varphi|\bigr)
\psi_1^{1/2}\bigr\|_{L_2}^{1-2/q} \|\varphi\psi_2^{1/2}\|_{L_2}^{2/q}. 
\end{multline*}
If $\varphi|_{y=0}=\varphi|_{y=L}=0$ one has to repeat this argument with the use of the second one of inequalities \eqref{1.19}.

If $k=2$, $m=1$, $q=2$ integration by parts yields an equality
\begin{multline*}
\iint (\varphi_x^2+\varphi_y^2)\psi_1^{1/2}\psi_2^{1/2}\, dxdy =
-\iint (\varphi_{xx}+\varphi_{yy})\psi_1^{1/2}\cdot \varphi\psi_2^{1/2}\,dxdy \\
-\iint\varphi\varphi_x(\psi_1^{1/2}\psi_2^{1/2})'\,dxdy
+\int_{\mathbb R} \left(\varphi\varphi_y\psi_1^{1/2}\psi_2^{1/2}\right)\Big|_{y=0}^{y=L}\,dx,
\end{multline*}
which in the case $\varphi|_{y=0}=\varphi|_{y=L}=0$ provides \eqref{1.18} with the constant $c$ independent on $L$, while in the general case one has to use the one-dimensional interpolating inequality
$$
\sup\limits_{y\in (0,L)} |\psi| \leq c\Bigr[ \Bigl(\int_0^L (\psi')^2\,dy\Bigr)^{1/4}
\Bigl(\int_0^L \psi^2\,dy\Bigr)^{1/4}+\frac{1}{L^{1/2}}\Bigl(\int_0^L \psi^2\,dy\Bigr)^{1/2}\Bigr].
$$

Combination of the already obtained inequalities \eqref{1.18} in the cases $k=1$, $m=0$ and $k=2$, $m=1$, $q=2$ obviously provides this inequality also in the case $k=2$, $m=0$, $q<+\infty$ since 
$$
|\varphi|^q \psi_1^{q/4-1/2}\psi_2^{q/4+1/2} = |\varphi|^q (\psi_1^{1/2}\psi_2^{1/2})^{q/2-1}\psi_2.
$$
\end{proof}

For the decay results we need Steklov inequality in such a form: for $\psi\in H_0^1(0,L)$
\begin{equation}\label{1.20}
\int_0^L \psi^2(y)\,dy \leq \frac{L^2}{\pi^2} \int_0^L \bigl(\psi'(y)\bigr)^2\,dy.
\end{equation}

The paper is organized as follows. Auxiliary linear problems are considered in Section~\ref{S2}. Section~\ref{S3} is devoted to existence of weak solutions to the original problem. Results on uniqueness are proved in Section~\ref{S4}. Decay of solutions is studied in Section~\ref{S5}.

\section{Auxiliary linear problems}\label{S2}

Consider a linear equation 
\begin{equation}\label{2.1}
u_t+bu_x+u_{xxx}+u_{xyy}+\delta(u_{xxxx}+u_{yyyy})=f(t,x,y)
\end{equation}
for a certain constant $\delta>0$. Besides initial condition \eqref{1.2} set boundary conditions
\begin{equation}\label{2.2}
u\big|_{y=0}=u\big|_{y=L}=u_{yy}\big|_{y=0}=u_{yy}\big|_{y=L}=0, \qquad (t,x)\in (0,T)\times\mathbb R,
\end{equation}
and consider the corresponding initial-boundary value problem in $\Pi_T$.

Introduce certain additional function spaces. Let $\EuScript S(\overline{\Sigma})$ be a space of infinitely smooth in $\overline{\Sigma}$ functions $\varphi(x,y)$ such that $\displaystyle{(1+|x|)^n|\partial^k_x\partial^l_y\varphi(x,y)|\leq c(n,k,l)}$ for any integer non-negative $n, k, l$ and all $(x,y)\in \overline{\Sigma}$. Let $\EuScript S_{exp}(\overline{\Sigma}_{\pm})$ denote a space of infinitely smooth in $\overline{\Sigma}_{\pm}$ functions $\varphi(x,y)$ such that $e^{n|x|}|\partial ^k_x\partial^l_y\varphi(x,y)|\leq c(n,k,l)$ for any integer non-negative $n, k, l$ and all $(x,y)\in \overline{\Sigma}_\pm$.

\begin{lemma}\label{L2.1}
Let $u_0\in \EuScript S(\overline{\Sigma})\cap \EuScript S_{exp}(\overline{\Sigma}_+)$, $f\in C^\infty\bigl([0,T]; \EuScript S(\overline{\Sigma})\cap \EuScript S_{exp}(\overline{\Sigma}_+)\bigr)$ and for any integer $j\geq 0$
$$
\partial^{2j}_y u_0\big|_{y=0}=\partial^{2j}_y u_0\big|_{y=L}=0, \qquad 
\partial^{2j}_y f\big|_{y=0}=\partial^{2j}_y f\big|_{y=L}=0.
$$
Then there exists a unique solution to problem \eqref{2.1}, \eqref{1.2}, \eqref{2.2} $u\in C^\infty\bigl([0,T]; \EuScript S(\overline{\Sigma})\cap \EuScript S_{exp}(\overline{\Sigma}_+)\bigr)$.
\end{lemma}

\begin{proof}
For any natural $l$ let $\psi_l(y)\equiv\sqrt{\frac{2}{L}}\sin{\frac{\pi l}{L} y}$, $\lambda_l=\left(\frac{\pi l}L\right)^2$.   
Then a solution to the considered problem can be written as follows:
$$
u(t,x,y)=\frac{1}{2\pi}\int_{\mathbb R}\, \sum_{l=1}^{+\infty}e^{i\xi x}\psi_l(y)\widehat{u}(t,\xi,l)\,d\xi,
$$
where
\begin{multline*}
\widehat{u}(t,\xi,l)\equiv \widehat{u_0}(\xi,l)e^{\left(i(\xi^3+\xi\lambda_l-b\xi)
-\delta(\xi^4+\lambda_l^2)\right)t} \\+
\int_0^t\widehat{f}(\tau,\xi,l)e^{\left(i(\xi^3+\xi\lambda_l-b\xi)-\delta(\xi^4+\lambda_l^2)\right)(t-\tau)}\,d\tau,
\end{multline*}
$$
\widehat{u_0}(\xi,l)\equiv\iint e^{-i\xi x}\psi_l(y)u_0(x,y)\,dxdy,\quad
\widehat f(t,\xi,l)\equiv\iint e^{-i\xi x}\psi_l(y) f(t,x,y)\,dxdy,
$$
and, obviously, $u\in C^\infty([0,T],\EuScript S(\overline{\Sigma}))$.

Next, let $v\equiv \partial^k_x\partial^l_y u$ for some integer $k, l\geq 0$. Then the function $v$ satisfies an equation of \eqref{2.1} type, where $f$ is replaced by $\partial^k_x\partial^l_y f$. Let $m\geq 4$. Multiplying this equation by $2x^m v$ and integrating over $\Sigma_+$ we derive an inequality 
\begin{multline*}
\frac{d}{dt}\iint_{\Sigma_+}x^m v^2\, dxdy 
+2\delta \iint_{\Sigma_+} x^m (v_{xx}^2+v^2_{yy})\,dxdy \\
-4\delta m(m-1) \iint_{\Sigma_+} x^{m-2} vv_{xx}\,dxdy 
\leq m(m-1)(m-2)\iint_{\Sigma_+}x^{m-3}v^2 \,dxdy  \\
+2\iint_{\Sigma_+}x^m \partial^k_x\partial^l_y fv\,dxdy 
+b m\iint_{\Sigma_+}x^{m-1} v^2\,dxdy.
\end{multline*}
Here
\begin{multline*}
\Bigl| 4\delta m(m-1) \iint_{\Sigma_+} x^{m-2} vv_{xx}\,dxdy\Bigr| \leq 2\delta \iint_{\Sigma_+} x^{2m}v_{xx}^2\,dxdy\\
+2\delta m^2(m-1)^2 \iint_{\Sigma_+} x^{m-4} v^2\,dxdy
\end{multline*}
and since $m(m-1)\leq 6(m-2)(m-3)$ for $m\geq 4$
\begin{multline*}
\frac{d}{dt}\iint_{\Sigma_+}x^m v^2\, dxdy \leq 12\delta m(m-1)(m-2)(m-3) \iint_{\Sigma_+} x^{m-4} v^2\,dxdy\\
+m(m-1)(m-2)\iint_{\Sigma_+}x^{m-3}v^2 \,dxdy 
+2\iint_{\Sigma_+}x^m \partial^k_x\partial^l_y fv\,dxdy \\
+|b|m\iint_{\Sigma_+}x^{m-1} v^2\,dxdy.
\end{multline*}

Let $\alpha>0$, $n\geq 4$. For any $m\in [4,n]$ multiplying the corresponding inequality by $\alpha^m/(m!)$ and summing by $m$ we obtain that for
$$
z_n(t)\equiv \iint_{\Sigma_+}\sum_{m=0}^n\frac{(\alpha x)^m}{m!}v^2(t,x,y)\,dxdy
$$
inequalities
$$
z_n'(t)\leq c z_n(t)+c, \quad z_n(0)\leq c
$$
hold uniformly with respect to $n$,
whence it follows that
$$
\sup_{t\in[0,T]}\iint_{\Sigma_+}e^{\alpha x} v^2\,dxdy<\infty.
$$
Thus, $u\in C^\infty([0,T], \EuScript S_{exp}(\overline{\Sigma}_+))$. 
\end{proof}

Next, consider generalized solutions. Let $u_0\in \bigl(\EuScript S(\overline{\Sigma})\cap \EuScript S_{exp}(\overline{\Sigma}_-)\bigr)'$, $f \in \bigl(C^\infty([0,T]; \EuScript S(\overline{\Sigma})\cap \EuScript S_{exp}(\overline{\Sigma}_-))\bigr)'$.

\begin{definition}\label{D2.1}
A function $u \in \bigl(C^\infty([0,T]; \EuScript S(\overline{\Sigma})\cap \EuScript S_{exp}(\overline{\Sigma}_-))\bigr)'$ is called a generalized solution to problem \eqref{2.1}, \eqref{1.2}, \eqref{2.2}, if for any function $\phi \in C^\infty\bigl([0,T]; \EuScript S(\overline{\Sigma})\cap \EuScript S_{exp}(\overline{\Sigma}_-)\bigr)$, such that $\phi|_{t=T}=0$ and $\phi|_{y=0}=\phi|_{y=L}=\phi_{yy}|_{y=0}=\phi_{yy}|_{y=L}=0$, the following equality holds:
\begin{equation}\label{2.3}
\langle u,\phi_t+b\phi_x+\phi_{xxx}+\phi_{xyy}-\delta\phi_{xxxx}-\delta\phi_{yyyy}\rangle +
\langle f,\phi\rangle +\langle u_0,\phi|_{t=0}\rangle=0.
\end{equation}
\end{definition}

\begin{lemma}\label{L2.2}
A generalized solution to problem \eqref{2.1}, \eqref{1.2}, \eqref{2.2} is unique.
\end{lemma}

\begin{proof} 
The proof is implemented by standard H\"olmgren's argument on the basis of Lemma~\ref{L2.1}.
\end{proof}

\begin{lemma}\label{L2.3}
Let $u_0\in L_2^{\psi(x)}$, $f\equiv f_0+f_{1\,x}+f_{2\,y}$, where $f_0\in L_1(0,T; L_2^{\psi(x)})$, $f_1,f_2\in L_2(0,T;L_2^{\psi(x)})$ for a certain admissible weight function $\psi(x)$. Then there exists a (unique) generalized solution to problem \eqref{2.1}, \eqref{1.2}, \eqref{2.2} $u\in C([0,T];L_2^{\psi(x)})\cap L_2(0,T;H^{2,\psi(x)})$. Moreover, for any $t_0\in (0,T]$ 
\begin{multline}\label{2.4}
\|u\|_{C([0,t_0];L_2^{\psi(x)})}+\|u\|_{ L_2(0,t_0;H^{2,\psi(x)})} \\
\leq c(T,\delta) \left[\|u_0\|_{L_2^{\psi(x)}}+\|f_0\|_{L_1(0,t_0;L_2^{\psi(x)})}
+\|f_1\|_{L_2(0,t_0;L_2^{\psi(x)})}
+\|f_2\|_{L_2(0,t_0;L_2^{\psi(x)})}\right].
\end{multline}
\end{lemma}

\begin{proof}
It is sufficient to consider smooth solutions from Lemma~\ref{L2.1} because of linearity of the problem.

Multiplying \eqref{2.1} by $2u(t,x,y)\psi(x)$ and integrating over $\Sigma$ we obtain an equality
\begin{multline}\label{2.5}
\frac{d}{dt}\iint u^2\psi \,dxdy
-b\iint u^2\psi'\,dxdy
+\iint (3u_x^2+u_y^2)\psi'\,dxdy
-\iint u^2\psi'''\,dxdy \\
+2\delta\iint(u_{xx}^2+u_{yy}^2)\psi \,dxdy 
-4\delta\iint u_x^2\psi''\,dxdy
+\delta\iint u^2\psi^{(4)}\,dxdy \\
=2\iint f_0 u\psi \,dxdy
-2\iint \bigl[f_1(u\psi)_x+f_2u_y\psi\bigr]\,dxdy, 
\end{multline}
whence \eqref{2.4} is immediate.
\end{proof}

\begin{lemma}\label{L2.4}
Let the hypothesis of Lemma~\ref{L2.3} be satisfied for $\psi(x)\geq 1$ $\forall x\in\mathbb R$. Then for any 
test function $\phi \in C^\infty\bigl([0,T]; \EuScript S(\overline{\Sigma})\cap \EuScript S_{exp}(\overline{\Sigma}_-)\bigr)$, such that $\phi|_{t=T}=0$, $\phi|_{y=0}=\phi|_{y=L}=0$, and for the corresponding generalized solution  $u\in C([0,T];L_2)\cap L_2(0,T;H^2)$ the following equality holds:
\begin{multline}\label{2.6}
\iiint_{\Pi_T} u(\phi_t+b\phi_x+\phi_{xxx}+\phi_{xyy})\,dxdydt
-\delta\iiint_{\Pi_T} (u_{xx}\phi_{xx}+u_{yy}\phi_{yy})\,dxdydt \\
+\iiint_{\Pi_T} (f_0\phi-f_1\phi_x-f_2\phi_y)\,dxdydt
+\iint u_0\phi\big|_{t=0}\,dxdy=0.
\end{multline}
\end{lemma}

\begin{proof}
Approximate the function $\phi$ by functions satisfying the hypothesis of Definition~\ref{D2.1}, write corresponding equality \eqref{2.3}, integrate by parts and then pass to the limit.
\end{proof}

\begin{lemma}\label{L2.5}
Let $u_0\in H^{1,{\psi(x)}}$, $f\equiv f_0+f_1$, where $f_0\in L_1(0,T;H^{1,{\psi(x)}})$, $f_1\in L_2(0,T;L_2^{\psi(x)})$ for a certain admissible weight function $\psi(x)$, $u_0|_{y=0}=u_0|_{y=l}=0$, $f_0|_{y=0}=f_0|_{y=L}= 0$. Then there exists a (unique) generalized solution to problem \eqref{2.1}, \eqref{1.2}, \eqref{2.2} $u\in C([0,T];H^{1,\psi(x)})\cap L_2(0,T;H^{3,\psi(x)})$. Moreover, for any $t_0\in (0,T]$
\begin{multline}\label{2.7}
\|u\|_{C([0,t_0];H^{1,\psi(x)})}+\|u\|_{ L_2(0,t_0;H^{3,\psi(x)})} \\
\leq c(T,\delta) \left[\|u_0\|_{H^{1,\psi(x)}}+\|f_0\|_{L_1(0,t_0;H^{1,\psi(x)})}
+\|f_1\|_{L_2(0,t_0;L_2^{\psi(x)})}\right].
\end{multline}
\end{lemma}

\begin{proof}
In the smooth case multiplying \eqref{2.1} by $-2\bigl((u_x(t,x,y)\psi(x)\bigr)_x
+u_{yy}(t,x,y)\psi(x)\bigr)$ and integrating over $\Sigma$ one obtains an equality
\begin{multline}\label{2.8}
\frac{d}{dt}\iint(u_x^2+u_y^2)\psi \,dxdy
-b\iint (u_x^2+u_y^2)\psi'\,dxdy
+\iint(3u_{xx}^2+4u^2_{xy}+u^2_{yy})\psi' \,dxdy \\
+2\delta\iint(u^2_{xxx}+u^2_{xxy}+u^2_{xyy}+u^2_{yyy})\psi \,dxdy 
-4\delta\iint (u^2_{xx}+u^2_{xy})\psi''\,dxdy \\
+\delta\iint(u_x^2+u_y^2)\psi^{(4)}\,dxdy 
-\iint(u^2_x+u^2_y)\psi'''\,dxdy 
=2\iint(f_{0\,x}u_x+f_{0\,y}u_y)\psi \,dxdy \\
-2\iint f_1[(u_x\psi)_x+u_{yy}\psi]\,dxdy,
\end{multline}
whence \eqref{2.7} follows.
\end{proof}

Now consider a linear initial-boundary value problem for an equation
\begin{equation}\label{2.9}
u_t+bu_x+u_{xxx}+u_{xyy}-\left(a_1(x,y)u_x\right)_x-\left(a_2(x,y)u_y\right)_y+a_0(x,y)u=f(t,x,y),
\end{equation}
where the functions $a_j$ at least satisfy \eqref{1.4}, \eqref{1.5}, with initial and boundary conditions \eqref{1.2}, \eqref{1.3}. Let $u_0\in \bigl(\EuScript S(\overline{\Sigma})\cap \EuScript S_{exp}(\overline{\Sigma}_-)\bigr)'$, $f \in \bigl(C^\infty([0,T]; \EuScript S(\overline{\Sigma})\cap \EuScript S_{exp}(\overline{\Sigma}_-))\bigr)'$.

\begin{definition}\label{D2.2}
A function $u \in \bigl(C^\infty([0,T]; \EuScript S(\overline{\Sigma})\cap \EuScript S_{exp}(\overline{\Sigma}_-))\bigr)'$ is called a generalized solution to problem \eqref{2.9}, \eqref{1.2}, \eqref{1.3}, if $a_1u_x, a_2u_y, a_0u \in \bigl(C^\infty([0,T]; \EuScript S(\overline{\Sigma})\cap \EuScript S_{exp}(\overline{\Sigma}_-))\bigr)'$ and for any function $\phi \in C^\infty\bigl([0,T]; \EuScript S(\overline{\Sigma})\cap \EuScript S_{exp}(\overline{\Sigma}_-)\bigr)$, such that $\phi|_{t=T}=0$ and $\phi|_{y=0}=\phi|_{y=L}=0$, the following equality holds:
\begin{multline}\label{2.10}
\langle u,\phi_t+b\phi_x+\phi_{xxx}+\phi_{xyy}\rangle-\langle a_1u_x,\phi_x\rangle - \langle a_2u_y,\phi_y\rangle \\
- \langle a_0u,\phi\rangle +\langle f,\phi\rangle +\langle u_0,\phi|_{t=0}\rangle=0.
\end{multline}
\end{definition}

\begin{lemma}\label{L2.6}
Let $a_0=a_1=a_2\equiv 0$, $u_0\in L_2$, $f\equiv f_0+f_{1\,x}+f_{2\,y}$, where $f_j\in L_2(\Pi_T)\ \forall j$. Assume that there exists a generalized solution to problem \eqref{2.9}, \eqref{1.2}, \eqref{1.3} $u\in L_\infty(0,T;L_2)\cap L_2(0,T;H^1)$. Then $u\in C([0,T];L_2)$ (after probable change on a set of zero measure) and for any $t\in [0,T]$
\begin{equation}\label{2.11}
\iint u^2(t,x,y)\,dxdy = \iint u_0^2\,dxdy +2\int_0^t\!\!\iint (f_0u-f_1u_x-f_2u_y)\,dxdyd\tau.
\end{equation}
\end{lemma}

\begin{proof}
Write equality \eqref{2.10}:
\begin{multline}\label{2.12}
\iiint_{\Pi_T} \bigl[u(\phi_t+b\phi_x+\phi_{xxx}+\phi_{xyy})+(f_0-f_1u_x-f_2u_y)\bigr]\,dxdydt \\
+\iint u_0\phi\big|_{t=0}\,dxdy=0.
\end{multline}
In particular, \eqref{2.12} provides that $u_t\in L_2(0,T;H^{-2})$ and, thus, after probable change on a set of zero measure $u\in C_w([0,T];L_2)$.

Substitute in \eqref{2.12} the variable $x$ by the variable $\xi$ and for any $x\in\mathbb R$ choose the test function $\phi$ in the form
$$
\phi(t,\xi,y)\equiv \frac 1h \lambda\left(\frac{x-\xi}h\right)\omega(t,y), 
$$
where $h>0$, $\lambda$ is the averaging kernel (for example, $\lambda(x)=ce^{1/(x^2-1)}$ if $|x|<1$, $\lambda(x)=0$ if $|x|\geq 1$, where the positive constant $c$ is chosen such that $\int_{-1}^1 \lambda(x)\,dx=1$), $\omega\in C^\infty([0,T]\times[0,L])$, $\omega|_{t=T}=\omega|_{t=0}= 0$, 
$\omega|_{y=0}=\omega|_{y=L}= 0$. Then if we set
$$
u^h(t,x,y)\equiv \frac 1h \int_{\mathbb R} \lambda\left(\frac{x-\xi}h\right) u(t,\xi,y)\,d\xi
$$
(and similarly $f_j^h$) we obtain an equality
\begin{equation}\label{2.13}
\iint_{(0,T)\times(0,L)} (u^h\omega_t-bu^h_x\omega-u^h_{xxx}\omega-u^h_x\omega_{yy}+f_0^h\omega+f^h_{1\,x}\omega
-f^h_2\omega_y)\,dydt=0.
\end{equation}
Note that $\partial_x^j u^h\in L_\infty(0,T;L_2)\cap L_2(0,T;H^1)\ \forall j$ and $u^h\in C_w([0,T];L_2)$.

Let $\phi\in C_0^\infty(\Pi_T)$. For any $x\in\mathbb R$ choose in equality \eqref{2.13} $\omega(t,y)\equiv \phi(t,x,y)$ and integrate over $\mathbb R$, then
$$
\iiint_{\Pi_T}\bigl[u^h(\phi_t+b\phi_x+\phi_{xxx}+\phi_{xyy}) +f_0^h\phi-f_1^h\phi_x-f_2^h\phi_y\bigr]\,dxdydt=0,
$$
that is in $\EuScript{D}'(\Pi_T)$ we have an equality
\begin{equation}\label{2.14}
u^h_t+bu^h_x+u^h_{xxx}+u^h_{xyy}=f_0^h+f^h_{1\,x}+f^h_{2\,y}.
\end{equation}
Note that $u^h_{xyy},f^h_{2\,y}\in L_2(0,T;H^{-1})$, $f_0^h,f^h_{1\,x}\in L_2(0,T;L_2)$, therefore, equality \eqref{2.14} yields that $u_t^h\in L_2(0,T;H^{-1})$. Since $u^h\in L_2(0,T;H^1_0)$ we have that $u^h\in C([0,T];L_2)$ and
\begin{multline*}
\|u^h(t,\cdot,\cdot)\|^2_{L_2}= \|u_0^h\|^2_{L_2}
+2\int_0^t \langle u^h_\tau, u^h\rangle\,d\tau  \\
=\|u_0^h\|^2_{L_2} +2\int_0^t\!\!\iint (f_0^hu^h-f_1^hu_x^h-f_2^hu_y^h-bu_x^hu^h-u^h_{xxx}u^h
+u^h_{xy}u^h_y|\,dxdyd\tau,
\end{multline*}
where the integrals of three last terms are obviously equal to zero.

Therefore, passing to the limit when $h\to 0$ we obtain equality \eqref{2.11}. In particular, \eqref{2.11} yields that 
$\|u(t,\cdot,\cdot)\|_{L_2}\in C[0,T]$ and taking into account weak continuity we find that $u\in C([0,T];L_2)$.
\end{proof}

\begin{lemma}\label{L2.7}
A generalized solution to problem \eqref{2.9}, \eqref{1.2}, \eqref{1.3} is unique in the space $L_\infty(0,T;L_2)\cap L_2(0,T;H^1)$ if assumptions \eqref{1.4}, \eqref{1.5} hold.
\end{lemma}

\begin{proof}
Consider a solution to problem \eqref{2.9}, \eqref{1.2}, \eqref{1.3} for $u_0\equiv 0$, $f\equiv 0$ as a solution to a problem of the same type but where $a_0=a_1=a_2\equiv 0$, $u_0\equiv 0$, $f\equiv f_0+f_{1\,x}+f_{2\,y}$, $f_0\equiv -a_0u$, $f_1\equiv a_1u_x$, $f_2\equiv a_2u_y$. Then by virtue of \eqref{2.11}
$$
\iint u^2(t,x,y)\,dxdy + 2\int_0^t\!\! \iint (a_1u_x^2+a_2u_y^2+a_0u^2)\,dxdyd\tau =0
$$
and, therefore, $u\equiv 0$.
\end{proof}

\begin{lemma}\label{L2.8}
Let assumptions \eqref{1.4}--\eqref{1.6} be satisfied, $u_0\in L_2^{\psi(x)}$, $f\equiv f_0+f_{1\,x}+f_{2\,y}$, where
$f_0\in L_1(0,T;L_2^{\psi(x)})$, $f_1,f_2\in L_2(0,T;L_2^{\psi(x)})$ for a certain admissible weight function $\psi(x)\geq 1 \ \forall x\in\mathbb R$ such that $\psi'(x)>0 \ \forall x\in\mathbb R$. Then there exists a (unique) generalized solution to problem \eqref{2.9}, \eqref{1.2}, \eqref{1.3} $u\in C_w([0,T];L_2^{\psi(x)})\cap 
L_2(0,T;H^{1,{\psi(x)}})$ and for any $t\in[0,T]$
\begin{multline}\label{2.15}
\iint u^2\psi(x)\,dxdy
+\int_0^t\!\! \iint |Du|^2\psi'\,dxdyd\tau
+2\int_0^t\!\! \iint (a_1u_x^2+a_2u_y^2)\psi\, dxdyd\tau \\
\leq \iint u_0^2\psi\,dxdy
+2\int_0^t\!\! \iint \bigl[f_0u\psi-f_1(u\psi)_x-f_2u_y\psi\bigr]\,dxdyd\tau \\
+c\int_0^t\!\! \iint u^2\psi\,dxdyd\tau,
\end{multline}
where the constant $c$ depends on $b$, $a_0$ and $\psi$.
\end{lemma}

\begin{proof}
For any $h\in (0,1]$ and $v\in Y(\Pi_T)=C([0,T];L_2^{\psi(x)})\cap L_2(0,T;H^{2,{\psi(x)}})$ consider an initial-boundary value problem for an equation
\begin{equation}\label{2.16}
u_t+bu_x+u_{xxx}+u_{xyy}+h(u_{xxxx}+u_{yyyy})=f+\left(a_1v_x\right)_x+\left(a_2v_y\right)_y-a_0v
\end{equation}
with initial and boundary conditions \eqref{1.2}, \eqref{2.2}. Lemma~\ref{L2.3} provides that there exists a solution to this problem $u_{h,v}\in Y(\Pi_T)$. By virtue of \eqref{2.4} and \eqref{1.18} for $v,\widetilde{v}\in Y(\Pi_T)$
\begin{multline*}
\|u_{h,v}-u_{h,\widetilde{v}}\|_{Y(\Pi_{t_0})} 
\leq c\left[ \|a_0(v-\widetilde{v})\|_{L_1(0,t_0;L_2^{\psi(x)})}
+ \|a_1(v-\widetilde{v})_x\|_{L_2(0,t_0;L_2^{\psi(x)})}\right. \\ \left.
+ \|a_2(v-\widetilde{v})_y\|_{L_2(0,t_0;L_2^{\psi(x)})}\right] 
\leq c_1 t_0^{1/4} \|v-\widetilde{v}\|^{1/2}_{C([0,t_0];L_2^{\psi(x)})}
 \|v-\widetilde{v}\|^{1/2}_{L_2(0,t_0;H^{2,{\psi(x)}})},
\end{multline*}
whence by the standard argument succeeds existence of a solution $u_h\in Y(\Pi_T)$ to an initial-boundary value problem for an equation
\begin{equation}\label{2.17}
u_t+bu_x+u_{xxx}+u_{xyy}+h(u_{xxxx}+u_{yyyy})-\left(a_1u_x\right)_x-\left(a_2u_y\right)_y+a_0u=f
\end{equation}
with initial and boundary conditions \eqref{1.2}, \eqref{2.2}.

Multiply \eqref{2.17} by $2u_h(t,x,y)\psi(x)$ and integrate over $\Sigma$, then (note that $u_{h\,t}\psi^{1/2}\in
L_2(0,T;H^{-2})$):
\begin{multline}\label{2.18}
\frac{d}{dt} \iint u_h^2\psi(x)\,dxdy
+\iint (3u_{h\,x}^2+u_{h\,y}^2)\psi'\,dxdy
+2h\iint (u_{h\,xx}^2+u_{h\,yy}^2)\psi\,dxdy \\
= \iint u_h^2(b\psi'+\psi'''-h\psi^{(4)})\,dxdy 
+4h \iint u_{h\,x}^2\psi''\,dxdy\\
+2\iint \left[\bigl(f_0u_h\psi-f_1(u_h\psi)_x-f_2u_{h\,y}\bigr)-\bigl(a_1u_{h\,x}(u_h\psi)_x+a_2u_{h\,y}^2\psi+a_0u_h^2\psi\bigr)\right]\,dxdy.
\end{multline}
Since for a certain constant $\widetilde{a}>0$, any $(x,y)\in\Sigma$ and $j=1$ or $2$
\begin{equation}\label{2.19}
\psi'(x)+2a_j(x,y)\psi(x)\geq \widetilde{a}\psi(x)
\end{equation}
inequality \eqref{2.18} yields that
\begin{multline}\label{2.20}
\iint u_h^2(t,x,y)\psi(x)\,dxdy
+\widetilde{a}\int_0^t\!\!\iint |Du_h|^2\psi\,dxdyd\tau \\
+h\int_0^t\!\!\iint (u_{h\,xx}^2+u_{h\,yy}^2)\psi\,dxdyd\tau 
\leq  \iint u_0^2\psi\,dxdy
+c\int_0^t\!\!\iint u_h^2\psi\,dxdyd\tau \\
+c\int_0^t\!\!\iint \bigl[|f_0u_h|+|f_1u_{h\,x}|+|f_1u_h|+|f_2u_{h\,y}|+|u_h u_{h\,x}|\bigr]\psi\,dxdyd\tau.
\end{multline}
Moreover, equality \eqref{2.6} yields that
\begin{multline}\label{2.21}
\iiint_{\Pi_T} u_h(\phi_t+b\phi_x+\phi_{xxx}+\phi_{xyy})\,dxdydt
-h\iiint_{\Pi_T} (u_{h\,xx}\phi_{xx}+u_{h\,yy}\phi_{yy})\,dxdydt \\
+\iiint_{\Pi_T} (f_0\phi-f_1\phi_x-f_2\phi_y)\,dxdydt
-\iiint_{\Pi_T} (a_0u_h\phi+a_1u_{h\,x}\phi_x+a_2u_{h\,y}\phi_y)\,dxdydt \\
+\iint u_0\phi\big|_{t=0}\,dxdy=0
\end{multline}
for any test function $\phi$ from Definition~\ref{D2.2}. With the use of \eqref{2.20} passing to the limit when $h\to +0$ in \eqref{2.18} and \eqref{2.21} we finish the proof.
\end{proof}

\begin{lemma}\label{L2.9}
Let assumptions \eqref{1.4}, \eqref{1.5} and \eqref{1.7} be satisfied, $u_0\in L_2^1$, $f\equiv f_0+f_{1\,x}+f_{2\,y}$, where $f_0\in L_1(0,T;L_2^1)$, $f_1,f_2\in L_2(0,T;L_2^{3/2})$. Then there exists a (unique) generalized solution to problem \eqref{2.9}, \eqref{1.2}, \eqref{1.3} $u\in C_w([0,T];L_2^1)\cap 
L_2(0,T;H^{1,1/2})$ and for any $t\in[0,T]$
\begin{multline}\label{2.22}
\iint u^2\rho_1(x)\,dxdy
+c_0\int_0^t\!\! \iint |Du|^2\rho_{1/2}\,dxdyd\tau 
\leq \iint u_0^2\rho_1\,dxdy \\
+c\int_0^t\!\! \iint u^2\rho_1\,dxdyd\tau 
+2\int_0^t\!\! \iint \left[f_0u\rho_1-f_1(u\rho_1)_x-f_2u_y\rho_1\right]\,dxdyd\tau,
\end{multline}
where the positive constants $c_0$ and $c$ depend on $b$ and properties of $a_j$.
\end{lemma}

\begin{proof}
As in the proof of Lemma~\ref{L2.8} consider for $h\in (0,1]$ initial-boundary value problems \eqref{2.17}, \eqref{1.2}, \eqref{2.2}. Lemma~\ref{L2.3} provides that solutions to these problems $u_h\in C_w([0,T];L_2^1)\cap
L_2(0,T;H^{2,1})$ exist. Multiplying \eqref{2.17} by $2u_h(t,x,y)\rho_1(x)$ and integrating over $\Sigma$ one obtains inequality \eqref{2.18} for $\psi\equiv \rho_1$. Here
\begin{multline*}
\iint (3u_{h\,x}^2+u_{h\,y}^2)\rho'_1\,dxdy
+2\iint (a_1u_{h\,x}^2\rho_1+a_2u_{h\,y}^2\rho_1)\,dxdy \\
\geq 2c_0 \iint|Du_h|^2\rho_{1/2}\,dxdy,
\end{multline*}
$$
\Bigl| \iint a_1u_{h\,x}u_h\rho'_1\,dxdy\Bigr| 
\leq \varepsilon \iint u^2_{h\,x}\rho_{1/2}\,dxdy
+c(\varepsilon)\iint u_h^2\rho_1\,dxdy,
$$
where $\varepsilon>0$ can be chosen arbitrarily small,
$$
\Bigl| \iint f_1u_{h\,x}\rho_1\,dxdy\Bigr| 
\leq \varepsilon \iint u^2_{h\,x}\rho_{1/2}\,dxdy
+c(\varepsilon)\iint f_1^2\rho_{3/2}\,dxdy,
$$
$\iint f_2u_{h\,y}\rho_1\,dxdy$ is estimated in a similar way. The end of the proof is the same as for Lemma~\ref{L2.8}.
\end{proof}

\section{Existence of weak solutions}\label{S3}

Let assumptions \eqref{1.4}, \eqref{1.5} be satisfied, $u_0\in L_2$, $f\in L_1(0,T;L_2)$.

\begin{definition}\label{D3.1}
A function $u\in L_\infty(0,T;L_2)\cap L_2\left(0,T;H^1((-r,r)\times(0,L))\right)$ $\forall r>0$ is called a weak solution to problem \eqref{1.1}--\eqref{1.3} if for any function $\phi\in C^\infty(\overline{\Pi}_T)$, such that 
$\phi|_{t=T}\equiv0$, $\phi(t,x,y)=0$ when $|x|\geq r$ for some $r>0$ and 
$\phi|_{y=0}=\phi|_{y=L}= 0$, the following equality holds:
\begin{multline}\label{3.1}
\iiint_{\Pi_T}\bigl[u(\phi_t+b\phi_x+\phi_{xxx}+\phi_{xyy} 
+\frac12 u^2\phi_{x}-a_1u_x\phi_x-a_2u_y\phi_y-a_0u\phi+f\phi\bigr]\,dxdyd\tau \\
+\iint_\Sigma u_0\phi\big|_{t=0}\,dxdy=0.
\end{multline}
\end{definition}

\begin{remark}\label{R3.1}
If a weak solution to problem \eqref{1.1}--\eqref{1.3} $u\in L_\infty(0,T; H^1)$ then equality \eqref{3.1} also holds for  test functions $\phi$ from Definition~\ref{D2.2}.
\end{remark}

Besides the original problem consider an auxiliary problem for a "regularized" equation 
\begin{equation}\label{3.2}
u_t+bu_x+u_{xxx}+u_{xyy}+\delta(u_{xxxx}+u_{yyyy})+\bigl(g(u)\bigr)_x-(a_1u_x)_x-(a_2u_y)_y+a_0u=f
\end{equation}
with initial condition \eqref{1.2} and boundary conditions \eqref{2.2} for $\delta>0$ and certain function $g$. A notion of a weak solution is introduced similarly to Definition~\ref{D3.1} but here the solution is assumed to belong to the space $L_\infty(0,T;L_2)\cap L_2(0,T;H^2)$, it is also assumed that 
$g(u(t,x,y))\in L_1\bigl((0,T)\times(-r,r)\times(0,L)\bigr) \ \forall r>0$, in the corresponding integral equality the term $u^2\phi_x/2$ is substituted by $g(u)\phi_x$ and the terms $-\delta (u_{xx}\phi_{xx}+u_{yy}\phi_{yy})$ are supplemented in the first integral. Note that if $g\equiv 0$ a weak solution to problem \eqref{3.2}, \eqref{1.2}, \eqref{2.2} satisfy equality \eqref{2.6} where $f_0\equiv f-a_0u$, $f_1\equiv a_1u_x$, $f_2\equiv a_2u_y$.

\begin{lemma}\label{L3.1}
Let $\delta>0$, $g\in C^1(\mathbb R)$, $g(0)=0$, $|g'(u)|\leq c\ \forall u\in\mathbb R$ and assumption \eqref{1.5} holds. Let $u_0\in L_2^{\psi(x)}$, $f\in L_1(0,T;L_2^{\psi(x)})$ for a certain admissible weight function $\psi(x)\geq 1\ \forall x\in\mathbb R$. Then problem \eqref{3.2}, \eqref{1.2}, \eqref{2.2} has a unique weak solution $u\in C([0,T];L_2^{\psi(x)})\cap L_2(0,T;H^{2,{\psi(x)}})$.
\end{lemma}

\begin{proof}
We apply the contraction principle. For $t_0\in(0,T]$ define a mapping $\Lambda$ on a set $Y(\Pi_{t_0})=C([0,t_0];L_2^{\psi(x)})\cap L_2(0,t_0;H^{2,{\psi(x)}})$ as follows: $u=\Lambda v\in Y(\Pi_{t_0})$ is a generalized solution to a linear problem
\begin{equation}\label{3.3}
u_t+bu_x+u_{xxx}+u_{xyy}+\delta (u_{xxxx}+ u_{yyyy})=f-\bigl(g(v)\bigr)_x +(a_1v_x)_x +(a_2v_y)_y -a_0v
\end{equation}
in $\Pi_{t_0}$ with initial and boundary conditions \eqref{1.2}, \eqref{2.2}.

Note that $|g(v)|\leq c|v|$ and, therefore,
$$
\|g(v)\|_{L_2(0,t_0;L_2^{\psi(x)})}\leq c\|v\|_{L_2(0,t_0;L_2^{\psi(x)})}<\infty.
$$
According to Lemma~\ref{L2.3} the mapping $\Lambda$ exists and according to Lemma~\ref{L2.4} the corresponding analogue of equality \eqref{2.6} can be written. Moreover, for functions $v,\widetilde{v}\in Y(\Pi_{t_0})$
$$
\|g(v)-g(\widetilde{v})\|_{L_2(0,t_0;L_2^{\psi(x)})}\leq c\|v-\widetilde{v}\|_{L_2(0,t_0;L_2^{\psi(x)})}\leq ct_0^{1/2}\|v-\widetilde{v}\|_{C([0,t_0];L_2^{\psi(x)})},
$$
\begin{multline*}
\|a_1(v_x-\widetilde{v}_x)\|_{L_2(0,t_0;L_2^{\psi(x)})} \\
\leq c\Bigl[\int_0^{t_0} \left(\bigl\||D^2(v-\widetilde{v})|\bigr\|_{L_2^{\psi(x)}} 
\|v-\widetilde{v}\|_{L_2^{\psi(x)}} + \|v-\widetilde{v}\|_{L_2^{\psi(x)}}^2\right)\,dt\Bigr]^{1/2} \\
\leq c_1t_0^{1/4} \|v-\widetilde{v}\|_{Y(\Pi_{t_0})}
\end{multline*}
and similarly for $a_2(v_y-\widetilde{v}_y)$.
As a result, according to inequality \eqref{2.4}
$$
\|\Lambda v-\Lambda\widetilde{v}\|_{Y(\Pi_{t_0})}\leq 
c(T,\delta)t_0^{1/4}\|v-\widetilde{v}\|_{Y(\Pi_{t_0})}.
$$
\end{proof}

Now we pass to the results of existence in Theorem~\ref{T1.1}.

\begin{proof}[Proof of existence part of Theorem~\ref{T1.1}] 
For $h\in (0,1]$ consider a set of initial-boundary value problems in $\Pi_T$
\begin{equation}\label{3.4}
u_t+bu_x+u_{xxx}+u_{xyy}+h(u_{xxxx}+u_{yyyy})+\left(g_h(u)\right)_x -(a_1u_x)_x-(a_2u_y)_y+a_0u=f
\end{equation}
with boundary conditions \eqref{1.2}, \eqref{2.2}, where
\begin{equation}\label{3.5}
g_h(u)\equiv\int_0^u\Bigl[\theta\eta(2-h|\theta|)+\frac{2\sgn\theta}{h}\eta(h|\theta|-1)\Bigr]\,d\theta.
\end{equation}
Note that $g_h(u)= u^2/2$ if $|u|\leq 1/h$, $|g'_h(u)|\leq 2/h\ \forall u\in\mathbb R$ and $|g'_h(u)|\leq 2|u|$ uniformly with respect to $h$. 

According to Lemma~\ref{L3.1} there exists a unique solution to this problem $u_h\in C([0,T];L_2^{\psi(x)})\cap L_2(0,T;H^{2,\psi(x)})$.

Next, establish appropriate estimates for functions $u_h$ uniform with respect to $h$.

Multiply \eqref{3.4} by $2u_h(t,x,y)$ and integrate over $\Sigma$, then (we omit the index $h$ in intermediate steps for simplicity):
\begin{multline}\label{3.6}
\frac{d}{dt} \iint u^2 \,dxdy
+2h\iint (u_{xx}^2+u_{yy}^2)\,dxdy +2\iint g'(u)uu_x\,dxdy \\
+2\iint (a_1u_x^2+a_2u_y^2+a_0u^2)\,dxdy =2\iint fu \,dxdy. 
\end{multline}

Since
\begin{equation}\label{3.7}
g'(u)uu_x=\Bigl(\int_0^u g'(\theta)\theta \,d\theta\Bigr)_x \equiv \bigl(g'(u)u\bigr)^*_x,
\end{equation}
where here and further $g^*(u)\equiv \int_0^u g(\theta)\,d\theta$ denotes the primitive for $g$ such that $g(0)=0$,
we have that $\iint g'(u)u_x u\,dxdy=0$ and equality \eqref{3.6} yields that 
\begin{equation}\label{3.8}
\|u_h\|_{C([0,T];L_2)}+h^{1/2}\|u_h\|_{L_2(0,T;H^2)}\leq c
\end{equation}
uniformly with respect to $h$ (and also uniformly with respect to $L$).

Next, multiply \eqref{3.4} by $2u_h(t,x,y)\psi(x)$ and integrate over $\Sigma$, then similarly to \eqref{2.18} with the use of \eqref{3.7} 
\begin{multline}\label{3.9}
\frac{d}{dt} \iint u^2\psi \,dxdy
+\iint (3u_x^2+u_y^2)\psi'\,dxdy
+2h\iint (u_{xx}^2+u_{yy}^2)\psi\,dxdy \\
-2\iint \left(g'(u)u\right)^*\psi'\,dxdy 
+2\iint (a_1u_x^2+a_2u_y^2)\psi\,dxdy 
+4h \iint u_x^2\psi''\,dxdy \\
=\iint u^2(b\psi'+\psi'''-h\psi^{(4)})\,dxdy
-\iint a_1u_x\psi'\,dxdy
+2\iint fu\psi \,dxdy. 
\end{multline}
Apply interpolating inequality \eqref{1.18} for $k=1$, $m=0$, $\psi_1=\psi_2\equiv\psi'$:
\begin{multline}\label{3.10}
\Bigl|\iint \left(g'(u)u\right)^*\psi'\,dxdy\Bigr| 
\leq \iint |u|^3\psi'\,dxdy \\
\leq \Bigl(\iint u^2\,dxdy\Bigr)^{1/2} \Bigl(\iint|u|^4(\psi')^2\,dxdy\Bigr)^{1/2} 
\leq c\Bigl(\iint u^2\,dxdy\Bigr)^{1/2} \\ \times \Bigl[\Bigl(\iint |Du|^2\psi'dxdy\Bigr)^{1/2} 
\Bigl(\iint u^2\psi' \,dxdy\Bigr)^{1/2}
+\iint u^2\psi' \,dxdy\Bigr]
\end{multline}
(note that here the constant $c$ is also uniform with respect to $L$).
Since the norm of the functions $u_h$ in the space $L_2$ is already  estimated in \eqref{3.8}, it follows from \eqref{3.9},
\eqref{3.10} that
\begin{equation}\label{3.11}
\|u_h\|_{C([0,T];L_2^{\psi(x)})}
+\bigl\| |Du_h| \bigr\|_{L_2(0,T;L_2^{\psi'(x)})}
+h^{1/2}\|u_h\|_{L_2(0,T;H^{2,\psi(x)})}\leq c.
\end{equation}

Finally, multiply \eqref{3.4} by $2u_h(t,x,y)\rho_0(x-x_0)$ for any $x_0\in\mathbb R$ and integrate over $\Sigma$, then it follows from the corresponding analogue of \eqref{3.9} that (see \eqref{1.10})
\begin{equation}\label{3.12}
\lambda (|Du_h|;T)\leq c.
\end{equation}

From equation \eqref{3.4} itself, \eqref{3.8} and the well-known embedding $L_1\subset H^{-2}$ it follows that uniformly with respect to $h$
\begin{equation}\label{3.13}
\|u_{h\,t}\|_{L_1(0,T;H^{-3})}\leq c.
\end{equation}

Estimates \eqref{3.11}--\eqref{3.13} by the standard argument provide existence of a weak solution to problem \eqref{1.1}--\eqref{1.3} $u\in X^{\psi(x)}(\Pi_T)$ (see, for example, \cite{BF13}) as a limit of functions $u_h$ when 
$h\to +0$.

If additional assumption \eqref{1.6} holds then inequalities \eqref{2.19} and \eqref{3.9} yield that uniformly with respect to $h$
\begin{equation}\label{3.14}
\bigl\| |Du_h|\bigr\|_{L_2(0,T;L_2^{\psi(x)})} \leq c
\end{equation}
and, therefore, $u\in L_2(0,T;H^{1,\psi(x)})$.

If additional assumption \eqref{1.7} holds then for $j=1$ or $2$
\begin{equation}\label{3.15}
\psi'(x)+2a_j(x,y)\psi(x)\geq \text{const}>0\qquad \forall\ x\leq 0
\end{equation}
and \eqref{3.9} provides that
\begin{equation}\label{3.16}
\bigl\| |Du_h|\bigr\|_{L_2(0,T;L_{2,-})} \leq c,
\end{equation}
therefore, $u\in L_2(0,T;H_-^1)$.

Finally, if additional assumption \eqref{1.8} holds then inequality \eqref{2.19} is valid for $x\geq 0$ so \eqref{3.9} provides that
\begin{equation}\label{3.17}
\bigl\| |Du_h|\bigr\|_{L_2(0,T;L_{2,+}^{\psi(x)})} \leq c
\end{equation}
and, therefore, $u\in L_2(0,T;H_+^{1,\psi(x)})$.
\end{proof}

\begin{remark}\label{R3.2}
Lemma~\ref{L2.6} provides that under the hypothesis of Theorem~\ref{T1.1} with additional assumption \eqref{1.6} and if $f\in L_2(\Pi_T)$ a weak solution to problem \eqref{1.1}--\eqref{1.3} $u\in C([0,T];L_2)$ (put it this Lemma $f_0\equiv f-a_0u$, $f_1\equiv a_1u_x-u^2/2$, $f_2\equiv a_2u_y$). 
\end{remark}

We now proceed to solutions in spaces $H^{1,\psi(x)}$ and firstly estimate a lemma analogous to Lemma~\ref{L3.1}.

\begin{lemma}\label{L3.2}
Let $\delta>0$, $g\in C^2(\mathbb R)$, $g(0)=0$ and $|g'(u)|,|g''(u)|\leq c \ \forall u\in\mathbb R$, $a_1,a_2\in W_\infty^1$, $a_0\in L_\infty$. Let $u_0\in H^{1,\psi(x)}$, $f\in L_1(0,T;H^{1,\psi(x)})$ for a certain admissible weight function $\psi(x)\geq 1\ \forall x\in\mathbb R$, $u_0|_{y=0}=u_0|_{y=L}= 0$, 
$f|_{y=0}=f|_{y=L}= 0$. Then problem \eqref{3.1}, \eqref{1.2}, \eqref{2.2} has a unique weak solution
 $u\in C([0,T];H^{1,\psi(x)})\cap L_2(0,T;H^{3,\psi(x)})$.
\end{lemma}

\begin{proof}
Introduce for $t_0\in (0,T]$ a space $Y_1(\Pi_{t_0})=C([0,t_0];H^{1,\psi(x)})\cap L_2(0,t_0;H^{3,\psi(x)})$ and define a mapping $\Lambda$ on it in the same way as in the proof of Lemma~\ref{L3.1} (with the substitution of $Y$ by $Y_1$). Since $|g'(v)v_x|\leq c|v_x|$ then
$$
\|g'(v)v_x\|_{L_2(0,t_0;L_2^{\psi(x)})}\leq ct_0^{1/2}\|v\|_{C([0,t_0];H^{1,\psi(x)})}.
$$
Besides that
\begin{multline*}
\|a_1v_{xx}\|_{L_2(0,t_0;L_2^{\psi(x)})}+\|a_2v_{yy}\|_{L_2(0,t_0;L_2^{\psi(x)})} \\
\leq ct_0^{1/4} \bigl\| |D^3v| \bigr\|_{L_2(0,t_0;L_2^{\psi(x)})}^{1/2}
\bigl\| |Dv| \bigr\|_{C([0,t_0];L_2^{\psi(x)})}^{1/2}
\end{multline*}
and according to Lemma~\ref{L2.5} (where $f_1\equiv -g'(v)v_x+(a_1v_x)_x+(a_2v_y)_y-a_0v$) such a mapping $\Lambda$ exists. Moreover, by virtue of \eqref{2.7}
\begin{equation}\label{3.18}
\|\Lambda v\|_{Y_1(\Pi_{t_0})}\leq c\left(\|u_0\|_{H^{1,\psi(x)}}+ 
\|f\|_{L_1(0,t_0;H^{1,\psi(x)})} +
t_0^{1/4}\|v\|_{Y_1(\Pi_{t_0})}\right).
\end{equation}
Next, since $|g'(v)v_x-g'(\widetilde{v})\widetilde{v}_x|\leq c|v_x|\cdot |v-\widetilde{v}|+c|v_x-\tilde{v}_x|$ inequality \eqref{2.7} provides that 
\begin{multline}\label{3.19}
\|\Lambda v-\Lambda\widetilde{v}\|_{Y_1(\Pi_{t_0})}   \leq 
c\bigl\|g'(v)v_x-g'(\widetilde v)\widetilde v_x\bigr\|_{L_2(0,t_0;L_2^{\psi(x)})}  \\ +
c\bigl\|\bigl(a_1(v_x-\widetilde v_x)\bigr)_x\bigr\|_{L_2(0,t_0;L_2^{\psi(x)})} +
c\bigl\|\bigl(a_2(v_y-\widetilde v_y)\bigr)_y\bigr\|_{L_2(0,t_0;L_2^{\psi(x)})} \\ +
c\bigl\|a_0(v-\widetilde v)\bigr\|_{L_2(0,t_0;L_2^{\psi(x)})}  \\ \leq
c\Bigl[\sup_{t\in [0,t_0]}||v_x||_{L_2^{\psi(x)}}\bigl\|\sup_{(x,y)\in\Sigma}
|v-\widetilde{v}|\bigr\|_{L_2(0,t_0)}+t_0^{1/4}\|v-\widetilde{v}\|_{Y_1(\Pi_{t_0})}\Bigr] \\
\leq c_1\Bigl[\sup_{t\in[0,t_0]}\|v\|_{H^{1,\psi(x)}}\cdot 
\sup_{t\in[0,t_0]}\|v-\widetilde{v}\|_{L_2}^{1/2}\cdot t_0^{1/4}\cdot
\|v-\widetilde{v}\|_{L_2(0,t_0;H^2)}^{1/2} \\
+t_0^{1/4}\|v-\widetilde{v}\|_{Y_1(\Pi_{t_0})}\Bigr]  \\
\leq c_2 t_0^{1/4} \bigl(1+\sup_{t\in [0,t_0]}\|v\|_{H^{1,\psi(x)}}\bigr)\|v-\widetilde{v}\|_{Y_1(\Pi_{t_0})},
\end{multline}
where we used inequality \eqref{1.18} for $q=+\infty$, $\psi_1=\psi_2\equiv 1$.

Now we prove the following a priori estimate: if $u\in Y_1(\Pi_{T'})$ is a solution to the considered problem for some $T'\in (0,T]$ then
\begin{equation}\label{3.20}
\|u\|_{C([0,T'];H^{1,\psi(x)})}\leq c(\|u_0\|_{H^{1,\psi(x)}}, 
\|f\|_{L_1(0,T';H^{1,\psi(x)})}).
\end{equation} 

Multiply \eqref{3.2} by $2\bigl[u(t,x,y)\psi(x)-\left(u_x(t,x,y)\psi(x)\right)_x-u_{yy}(t,x,y)\psi(x)\bigr]$ and integrate over $\Sigma$ then
\begin{multline}\label{3.21}
\frac{d}{dt}\iint(u^2+|Du|^2)\psi \,dxdy+2\delta\iint|D^3u|^2\psi \,dxdy \\ 
\leq 2\iint(fu+f_xu_x+f_yu_y)\psi \,dxdy
+c\iint(u^2+|Du|^2+|D^2u|^2)\psi \,dxdy.
\end{multline}
Since
$$
\iint |D^2u|^2\psi\,dxdy \leq \varepsilon \iint |D^3u|^2\psi\,dxdy
+c(\varepsilon)\iint |Du|^2\psi\,dxdy,
$$
where $\varepsilon>0$ can be chosen arbitrarily small, inequality \eqref{3.21} provides estimate \eqref{3.20}.

Inequalities \eqref{3.18}, \eqref{3.19} give an opportunity to construct a solution to the considered problem locally in time by the contraction while estimate \eqref{3.20} enables to extend it for the whole time segment $[0,T]$.
\end{proof}

\begin{proof}[Proof of existence part of Theorem~\ref{T1.2}] 
As in the proof of Theorem~\ref{T1.1} consider the set of "regularized" problems \eqref{3.4}, \eqref{1.2}, \eqref{2.2} and for their corresponding solutions $u_h\in C([0,T];H^{1,\psi(x)})\cap L_2(0,T;H^{3,\psi(x)})$ (note that $|g''_h(u)|\leq c\ \forall u\in\mathbb R)$ establish appropriate estimates uniform with respect to $h$. 

Multiply \eqref{3.4} by $-2\bigl[\left(u_{h\,x}(t,x,y)\psi(x)\right)_x +u_{h\,yy}(t,x,y)\psi(x) +g_h\bigl(u_h(t,x,y)\bigr)\psi(x)\bigr]$ and integrate over $\Sigma$ then (index $h$ is again omitted):
\begin{multline}\label{3.22}
\frac{d}{dt} \iint \bigl(u_x^2+u_y^2-2g^*(u)\bigr)\psi\,dxdy
-b\iint \left(u_x^2+u_y^2-2g^*(u)\right)\psi'\,dxdy \\
+\iint (3u_{xx}^2+4u_{xy}^2+u_{yy}^2)\psi'\,dxdy 
-\iint (u_x^2+u_y^2)\psi'''\,dxdy \\
+2h\iint (u_{xxx}^2+u_{xxy}^2+u_{xyy}^2+u_{yyy}^2)\psi\,dxdy
-4h\iint (u_{xx}^2+u_{xy}^2)\psi''\,dxdy \\
+h\iint (u_x^2+u_y^2)\psi^{(4)}\,dxdy
-2\iint g'(u)u_x^2\psi'\,dxdy
+2\iint g(u)(u_{xx}+u_{yy})\psi'\,dxdy \\
+2h\iint g'(u)(u_x u_{xxx}+u_y u_{yyy})\psi\,dxdy
+2h\iint g(u)u_{xxx}\psi'\,dxdy \\
+2\iint \left(g'(u)g(u)\right)^*\psi'\,dxdy 
+2\iint [a_1u_{xx}^2+(a_1+a_2)u_{xy}^2+a_2u_{yy}^2]\psi\,dxdy \\
-\iint\bigl((a_{1\,xx}+a_{1\,yy})u_x^2+(a_{2\,xx}+a_{2\,yy})u_y^2\bigr)\psi\,dxdy
+\int_{\mathbb R} \left(a_{2\,y}u_y^2\psi\right)\big|^{y=L}_{y=0}\,dx \\
-\iint(a_{1\,x}u_y^2+a_{2\,x}u_y^2-2a_{1\,y}u_x u_y)\psi'\,dxdy
-\iint a_1(u_x^2+u_y^2)\psi''\,dxdy \\
-2\iint g'(u)(a_1u_x^2+a_2u_y^2)\psi\,dxdy 
-2\iint a_1g(u)u_x\psi'\,dxdy \\
+2\iint a_0(u_x^2+u_y^2)\psi\,dxdy
+2\iint (a_{0\,x}u_x+a_{0\,y}u_y)u\psi\,dxdy \\
-2\iint a_0g(u)\psi\,dxdy
=2\iint \bigl(f_xu_x+f_yu_y-fg(u)\bigr)\psi\,dxdy.
\end{multline}
Remind that $g^*(u)=\int_0^u g(\theta)\,d\theta$, so $|g^*(u)|\leq |u|^3/3$ and similarly to \eqref{3.10}
$$
\Bigl| \iint g^*(u)\psi\,dxdy\Bigr| \leq c\Bigl(\iint |Du|^2\psi\,dxdy\Bigr)^{1/2}+c,
$$
where the already obtained estimated \eqref{3.11} on $\|u_h\|_{C([0,T];L_2^{\psi(x)})}$ is also used. Next, with the use of \eqref{3.11} and \eqref{1.18} first for $k=1$ we derive that
\begin{multline*}
\Bigl| \iint g'(u)u_x^2\psi'\,dxdy\Bigr| 
\leq 2\Bigl(\iint u^2\,dxdy\Bigr)^{1/2}\Bigl(\iint u_x^4(\psi')^2\,dxdy\Bigr)^{1/2} \\
\leq c \Bigl[\Bigl(\iint |Du_x|^2\psi'\,dxdy\Bigr)^{1/2}\Bigl(\iint u_x^2\psi\,dxdy\Bigr)^{1/2}
+\iint u_x^2\psi\,dxdy\Bigr],
\end{multline*}
\begin{multline*}
\Bigl|\iint g(u)(u_{xx}+u_{yy})\psi'\,dxdy\Bigr| 
\leq c\Bigl(\iint(u_{xx}^2+u_{yy}^2)\psi'\,dxdy\Bigr)^{1/2} \Bigl(\iint u^4\psi^2\,dxdy\Bigr)^{1/2} \\
\leq c_1\Bigl(\iint (u_{xx}^2+u_{yy}^2)\psi'\,dxdy\Bigr)^{1/2} \Bigl[\Bigl(\iint |Du|^2\psi\,dxdy\Bigr)^{1/2}+1\Bigr]
\end{multline*}
and then with the use of \eqref{1.18} for $k=3$ that
\begin{multline*}
h\Bigl|\iint g'(u)(u_x u_{xxx}+u_y u_{yyy})\psi\,dxdy\Bigr| 
\leq ch\Bigl(\iint |D^3u|^2\psi\,dxdy\Bigr)^{1/2}  \\ \times 
\Bigl(\iint u^4\psi^2\,dxdy\Bigr)^{1/4} 
\Bigl(\iint |Du|^4\psi^2\,dxdy\Bigr)^{1/4} 
\leq c_1h\Bigl(\iint |D^3u|^2\psi\,dxdy\Bigr)^{5/6}+c_1.
\end{multline*}

Next, we apply interpolating inequality \eqref{1.19} for $p=1$ and find that
\begin{multline*}
\iint a_1|g'(u)|u_x^2\psi\,dxdy 
\leq 2\Bigl(\iint a_1^2u_x^4\psi^2\,dxdy\Bigr)^{1/2} \Bigl(\iint u^2\,dxdy\Bigr)^{1/2} \\
\leq c \Bigl(\iint \left|D(a_1u_x^2\psi)\right|\,dxdy\Bigr) \\
\leq \varepsilon\iint a_1(u_{xx}^2+u_x^2+u_{xy}^2)\psi\,dxdy
+c(\varepsilon) \iint \left(|a_{1\,x}|+|a_{1\,y}|\right)u_x^2\psi\,dxdy,
\end{multline*}
where $\varepsilon>0$ can be chosen arbitrarily small. Of course, $\iint a_2|g'(u)|u_y^2\psi\,dxdy$ is estimated in the same way.

Finally,
\begin{multline*}
\iint|fg(u)|\psi\,dxdy \leq \Bigl(\iint f^2\psi\,dxdy\Bigr)^{1/2}\Bigl(\iint u^4\psi^2\,dxdy\Bigr)^{1/2} \\
\leq c \Bigl(\iint f^2\psi\,dxdy\Bigr)^{1/2} 
\Bigl[\Bigl(\iint |Du|^2\psi\,dxdy\Bigr)^{1/2}+1\Bigr].
\end{multline*}

Therefore, it follows from \eqref{3.22} that uniformly with respect to $h$
\begin{equation}\label{3.23}
\|u_h\|_{C([0,T];H^{1,\psi(x)})}
+\bigl\| |D^2u_h| \bigr\|_{L_2(0,T;L_2^{\psi'(x)})}
+h^{1/2}\|u_h\|_{L_2(0,T;H^{3,\psi(x)})}\leq c.
\end{equation}
As in the proof of Theorem~\ref{T1.1} one can repeat this argument where $\psi(x)$ is substituted by $\rho_0(x-x_0)$ and similarly to \eqref{3.12} obtain an estimate
\begin{equation}\label{3.24}
\lambda (|D^2u_h|;T)\leq c.
\end{equation}
Estimates \eqref{3.23}, \eqref{3.24} and \eqref{3.13} provide existence of a weak solution to the considered problem $u\in X^{1,\psi(x)}(\Pi_T)$.

Additional properties of this solution under assumptions \eqref{1.6}, \eqref{1.7} or \eqref{1.8} are established in the same way (see \eqref{3.14}--\eqref{3.17}) as in Theorem~\ref{T1.1}.
\end{proof}

\section{Uniqueness}\label{S4}

Consider four lemmas which provide uniqueness results of Theorems~\ref{T1.1} and~\ref{T1.2}. Note that similar argument can be applied to establish continuous dependence of solutions on the initial data $u_0$ and the function $f$.

\begin{lemma}\label{L4.1}
Let assumptions \eqref{1.4} and \eqref{1.5} hold, then a weak solution to problem \eqref{1.1}--\eqref{1.3} is unique in the space $X^{1,1/2}(\Pi_T)$.
\end{lemma}

\begin{proof}
Let $u$ and $\widetilde{u}$ be two solutions to the same problem in the considered space, $v\equiv u-\widetilde{u}$. Then $v$ is a weak solution to a linear problem
\begin{gather}\label{4.1}
v_t+bv_x+v_{xxx}+v_{xyy}-(a_1v_x)_x-(a_2v_y)_y+a_0v=\frac12\left(\widetilde{u}^2-u^2\right)_x, \\
\label{4.2}
v\big|_{t=0}=0,\qquad v\big|_{y=0}=v\big|_{y=L}=0.
\end{gather}

Let $\psi(x)\equiv \kappa_{1/2}(x)$. Then $v\in L_\infty(0,T;H^{1,\psi(x)})$. Note also that $u^2,\widetilde{u}^2\in L_\infty(0,T;L_2^1)$ since
\begin{equation}\label{4.3}
\iint u^4\rho_1(x)\,dxdy \leq c\Bigl(\iint\left(|Du|^2+u^2\right)\rho_{1/2}(x)\,dxdy\Bigr)^2<\infty.
\end{equation}
For $r\geq 1$ let $\psi_r(x)\equiv \psi(x)\eta(r+1-x)+(2+r)\eta(x-r)$, then $|D^2 v| \in L_2(0,T;L_2^{\psi_r(x)})$ and it follows from \eqref{4.1} that $v_t\psi_r^{1/2}\in L_2(0,T;H^{-1})$. 

Multiply \eqref{4.1} by $2v(t,x,y)\psi_r(x)$, integrate over $\Pi_t$ and then pass to the limit when $r\to +\infty$:
\begin{multline}\label{4.4}
\iint v^2\psi\,dxdy
+\int_0^t\!\!\iint(3v_x^2+v_y^2)\psi'\,dxdyd\tau
-\int_0^t\!\!\iint v^2(\psi'''+b\psi')\,dxdyd\tau \\
+2\int_0^t\!\!\iint(a_1v_x^2+a_2v_y^2+a_0v^2)\psi\,dxdyd\tau
+2\int_0^t\!\!\iint a_1vv_x\psi'\,dxdyd\tau \\
= 2\int_0^t\!\!\iint (\widetilde{u}\widetilde{u}_x-uu_x)v\psi\,dxdyd\tau.
\end{multline}
Here
$$
\Bigl|\iint (\widetilde{u}\widetilde{u}_x-uu_x)v\psi\,dxdy\Bigr| 
\leq c \iint\bigl(|u_x|+|\widetilde{u}_x|+|u|+|\widetilde{u}|\bigr)v^2\psi\,dxdy
$$
and since $\psi(x)/\psi'(x)\leq c(\varepsilon) (1+x_+)$ for all $x\in\mathbb R$
\begin{multline}\label{4.5}
\iint\bigl(|u_x|+|u|\bigr)v^2\psi\,dxdy
\leq \Bigl(2\iint(u_x^2+u^2)\frac{\psi}{\psi'}\,dxdy\Bigr)^{1/2}
\Bigl(\iint v^4\psi'\psi\,dxdy\Bigr)^{1/2} \\
\leq c\Bigl(\iint(u_x^2+u^2)\rho_{1/2}\,dxdy\Bigr)^{1/2}
\Bigl[\Bigl(\iint |Dv|^2\psi'\,dxdy\Bigr)^{1/2} \Bigl(\iint v^2\psi\,dxdy\Bigr)^{1/2} \\
+\iint v^2\psi\,dxdy\Bigr].
\end{multline}
Therefore, inequality \eqref{4.4} yields that $v\equiv 0$.
\end{proof}

\begin{lemma}\label{L4.2}
Let assumptions \eqref{1.4}--\eqref{1.6} hold, then a weak solution to problem \eqref{1.1}--\eqref{1.3} is unique in the space $L_\infty(0,T;L_2)\cap L_2(0,T;H^1)$.
\end{lemma}

\begin{proof}
As in the proof of Lemma~\ref{L4.1} consider linear problem \eqref{4.1}, \eqref{4.2}. Here similarly to \eqref{4.3}
$u^2,\widetilde{u}^2\in L_2(\Pi_T)$. Therefore, the hypothesis of Lemma~\ref{L2.8} is satisfied for $\psi(x)\equiv \rho_0(x)$ and according to \eqref{2.15}
\begin{multline}\label{4.6}
\iint v^2(t,x,y)\rho_0\,dxdy 
+\int_0^t\!\!\iint |Dv|^2\rho_0'\,dxdyd\tau
+2\int_0^t\!\!\iint (a_1v_x^2+a_2v_y^2)\rho_0\,dxdyd\tau \\
\leq c \int_0^t\!\!\iint v^2\rho_0\,dxdyd\tau+
\int_0^t\!\!\iint (u+\widetilde{u})v(v\rho_0)_x\,dxdyd\tau.
\end{multline}
Here
\begin{multline*}
\iint |u v v_x|\rho_0\,dxdy 
\leq 2\Bigl(\iint u^4\,dxdy\Bigr)^{1/4} \Bigl(\iint v^4\,dxdy\Bigr)^{1/4} 
\Bigl(\iint v_x^2\,dxdy\Bigr)^{1/2}\\
\leq c\Bigl(\iint |Du|^2\,dxdy\Bigr)^{1/4} \Bigl(\iint u^2\,dxdx\Bigr)^{1/4}
\Bigl(\iint |Dv|^2\,dxdy\Bigr)^{3/4} \Bigl(\iint v^2\,dxdy\Bigr)^{1/4} \\
\leq \varepsilon \iint |Dv|^2\,dxdy +c(\varepsilon)\iint |Du|^2\,dxdy
\iint v^2\,dxdy,
\end{multline*}
where $\varepsilon>0$ can be chosen arbitrarily small, and with the use of inequality \eqref{2.19} for $\psi=\rho_0$
we derive from \eqref{4.6} that $v\equiv 0$.
\end{proof}

\begin{lemma}\label{L4.3}
Let assumptions \eqref{1.4}, \eqref{1.5} and \eqref{1.7} hold, then a weak solution to problem \eqref{1.1}--\eqref{1.3} is unique in the space $L_\infty(0,T;L_2^1)\cap L_2(0,T;H^{1,1/2})$.
\end{lemma}

\begin{proof}
Here
\begin{multline*}
\int_0^T\!\!\iint u^4\rho_{3/2}\,dxdydt 
\leq c\int_0^T\!\!\iint u^4\rho^{1/2}\rho_1\,dxdydt \\
\leq c_1 \int_0^T \Bigl[ \iint |Du|^2\rho^{1/2}\,dxdy \iint u^2\rho_1\,dxdy
+\Bigl(\iint u^2\rho_1\,dxdy\Bigr)^2\Bigr]\,dt <\infty
\end{multline*}
and by virtue of Lemma~\ref{L2.9} for the function $v$
\begin{multline}\label{4.7}
\iint v^2(t,x,y)\rho_1\,dxdy 
+c_0\int_0^t\!\!\iint |Dv|^2\rho_{1/2}\,dxdyd\tau \\
\leq c \int_0^t\!\!\iint v^2\rho_1\,dxdyd\tau+
\int_0^t\!\!\iint (u+\widetilde{u})v(v\rho_1)_x\,dxdyd\tau.
\end{multline}
Since
\begin{multline*}
\iint |u v v_x|\rho_1\,dxdy  \leq 
c\iint |u|\rho_{1/2}^{1/4}\rho_1^{1/4} \cdot |v|\rho_{1/2}^{1/4}\rho_1^{1/4} \cdot |v_x|\rho_{1/2}^{1/2}\,dxdy \\
\leq c\Bigl(\iint u^4\rho_{1/2}\rho_1\,dxdy\Bigr)^{1/4} \Bigl(\iint v^4\rho_{1/2}\rho_1\,dxdy\Bigr)^{1/4} 
\Bigl(\iint v_x^2\rho_{1/2}\,dxdy\Bigr)^{1/2} \\
\leq c_1\Bigr[\Bigl(\iint |Du|^2\rho_{1/2}\,dxdy\Bigr)^{1/4} \Bigl(\iint u^2\rho_1\,dxdx\Bigr)^{1/4}
+ \Bigl(\iint u^2\rho_1\,dxdx\Bigr)^{1/2}\Bigr] \\
\times \Bigl[\Bigl(\iint |Dv|^2\rho_{1/2}\,dxdy\Bigr)^{1/4} \Bigl(\iint v^2\rho_1\,dxdy\Bigr)^{1/4}
+ \Bigl(\iint v^2\rho_1\,dxdy\Bigr)^{1/2}\Bigr] \\
\times \Bigl(\iint v_x^2\rho_{1/2}\,dxdy\Bigr)^{1/2} 
\leq \varepsilon \iint |Dv|^2\rho_{1/2}\,dxdy \\
+c(\varepsilon)\Bigl[\iint |Du|^2\rho_{1/2}\,dxdy+1\Bigr]
\iint v^2\rho_1\,dxdy,
\end{multline*}
where $\varepsilon>0$ can be chosen arbitrarily small, inequality \eqref{4.7} yields that $v\equiv 0$.
\end{proof}

\begin{lemma}\label{L4.4}
Let assumptions \eqref{1.4}, \eqref{1.5} and \eqref{1.8} hold, then a weak solution to problem \eqref{1.1}--\eqref{1.3} is unique in the space $X^{1,0}(\Pi_T)\cap L_2(0,T;H^2_+)$.
\end{lemma}

\begin{proof}
Repeat the argument of the proof of Lemma~\ref{L4.1} for the function $\psi(x)\equiv \kappa_0(x)$. Here $v_t\psi^{1/2}\in L_2(0,T;H^{-1})$, therefore, one does not have to introduce the function $\psi_r$ and can multiply \eqref{4.1} directly by $2v(t,x,y)\psi(x)$. Instead of \eqref{4.3}, \eqref{4.5} we have that $u^2,\widetilde{u}^2\in L_\infty(0,T;L_2)$ and
\begin{multline}\label{4.8}
\iint(|u_x|_+|u|)v^2\psi\,dxdy
\leq \Bigl(2\iint(u_x^2+u^2)\,dxdy\Bigr)^{1/2}
\Bigl(\iint v^4\psi^2\,dxdy\Bigr)^{1/2} \\
\leq c\Bigl[\Bigl(\iint |Dv|^2\psi\,dxdy\Bigr)^{1/2} \Bigl(\iint v^2\psi\,dxdy\Bigr)^{1/2} 
+\iint v^2\psi\,dxdy\Bigr].
\end{multline}
Since for $j=1$ or $2$
\begin{equation}\label{4.9}
\kappa'_0(x)+2a_j(x,y)\kappa_0(x) \geq c\kappa_0(x) \quad
\forall\ (x,y)\in\Sigma,	
\end{equation}
it follows from \eqref{4.4} and \eqref{4.8} that $v\equiv 0$.
\end{proof}

\section{Long-time decay}\label{S5}

\begin{proof}[Proof of Corollary~\ref{C1.1}]
Equalities \eqref{3.6}, \eqref{3.7} provide that for the solutions $u_h$ to problem \eqref{3.4}, \eqref{1.2}, \eqref{2.2} the following inequality holds:
\begin{equation}\label{5.1}
\frac{d}{dt} \iint u_h^2\,dxdy
+2\iint (a_1u_{h\,x}^2+a_2u_{h\,y}^2+a_0u_h^2)\,dxdy \leq 0.
\end{equation}
With the use of \eqref{1.11}, \eqref{1.12} and \eqref{1.20} we derive that
$$
\iint(a_2u_{h\,y}^2+a_0u_h^2)\,dxdy \geq \beta \iint u_h^2\,dxdy
$$
and it follows from \eqref{5.1} that
$$
\|u_h(t,\cdot,\cdot)\|_{L_2} \leq e^{-\beta t}\|u_0\|_{L_2} \quad\forall t\geq 0.
$$
Passing to the limit when $h\to +0$ we obtain \eqref{1.13}.
\end{proof}

\begin{proof}[Proof of Corollaries~\ref{C1.2}--\ref{C1.5}]
First of all note that equality \eqref{5.1} provides that
\begin{equation}\label{5.2}
\|u_h(t,\cdot,\cdot)\|_{L_2} \leq \|u_0\|_{L_2} \quad\forall t\geq 0.
\end{equation}

Next, consider inequality \eqref{3.9} for $f\equiv 0$ and $\psi(x)\equiv\rho_0(\alpha x)$ for Corollary~\ref{C1.2},
$\psi(x)\equiv e^{2\alpha x}$ for Corollary~\ref{C1.3}, $\psi(x)\equiv 1+e^{2\alpha x}$ for Corollary~\ref{C1.4},
$\psi(x)\equiv \kappa_0(\alpha x)$ for Corollary~\ref{C1.5}, where $\alpha\in (0,1]$. Continuing inequality \eqref{3.10} we find with the use of \eqref{5.2} that uniformly with respect to $L$ and $\alpha$
\begin{multline*}
\Bigl|\iint\left(g_h'(u_h)u_h\right)^*\psi'\,dxdy\Bigr| 
\leq \frac12 \iint |Du_h|^2\psi'\,dxdy \\+
c\alpha\left(\|u_0\|_{L_2}+\|u_0\|^2_{L_2}\right) \iint u_h^2\psi\,dxdy.
\end{multline*}
Inequalities \eqref{2.19}, \eqref{4.9} and \eqref{1.20} yield that for a certain independent on $L$ and $\alpha$ constant $c_0>0$ 
$$
\iint\Bigl(\frac12 u^2_{h\,y}\psi'+2a_2u^2_{h\,y}\psi\Bigr)\,dxdy \geq \frac{c_0\alpha}{L^2} \iint u_h^2\psi\,dxdy.
$$
Therefore, it follows from \eqref{3.9} that uniformly with respect to $L$ and $\alpha$
\begin{multline*}
\frac{d}{dt} \iint u_h^2\psi\,dxdy + \frac{c_0\alpha}{L^2} \iint u_h^2\psi\,dxdy \\
\leq c\alpha\left(b+2\alpha+\|u_0\|_{L_2}+\|u_0\|^2_{L_2}\right) \iint u_h^2\psi\,dxdy
\end{multline*}
and choosing $L_0=\sqrt{c_0/(2cb)}$ if $b>0$, $\alpha_0$ and $\epsilon_0$ satisfying $4c(2\alpha_0+\epsilon_0+\epsilon_0^2)\leq c_0L^{-2}$, $\beta=c_0\alpha L^{-2}/8$ we derive an inequality
$$
\|u_h(t,\cdot,\cdot)\psi^{1/2}\|_{L_2} \leq e^{-\beta t}\|u_0\psi^{1/2}\|_{L_2} \quad\forall t\geq 0
$$
whence Corollaries~\ref{C1.2} if $b\leq 0$ and~\ref{C1.3}--\ref{C1.5} follows. For Corollary~\ref{C1.2} if $b>0$ note that this case can be reduced to the case $b=0$ by the substitution $\widetilde u(t,x,y)= u(t,x+bt,y)$.
\end{proof}


\begin{thebibliography}{30}

\bibitem{ABS} C.J.~Amick, J.L.~Bona and M.E.~Schonbek, {\it Decay of solutions of some nonlinear wave equations}, J. Differential Equ., {\bf 81}(1989), 1--49.

\bibitem{BF13} E.S.~Baykova and A.V.~Faminskii, {\it On initial-boundary-value problems in a strip for the generalized two-dimensional Zakharov--Kuznetsov equation}, Adv. Differential Equ., {\bf 18}(2013), 663--686.

\bibitem{CCFN} M.M.~Cavalcanti, V.N.~Domingos~Cavalcanti, A.~Faminskii and F.~Natali, {\it Decay of solutions to damped Korteweg--de~Vries type equation}, Appl. Math. Optim., {\bf 65}(2012), 221--251.

\bibitem{DL} G.G.~Doronin and N.A.~Larkin, {\it Stabilization of regular solutions for the Zakharov--Kuznetsov equation posed on bounded rectangles and on a strip}, 25 Sep. 2012, arXiv: 1209.5767v1 [math.AP].

\bibitem{F89} A.V.~Faminskii, {\it The Cauchy problem for quasilinear equations of odd order}, Mat. Sb., {\bf 180}(1989), 1183--1210. English transl. in Math. USSR-Sb., {\bf 68}(1991), 31--59.

\bibitem{F95} A.V.~Faminskii {\it The Cauchy problem for the Zakharov--Kuznetsov equation}, Differ. Uravn., {\bf 31}(1995), 1070--1081. English transl. in Differential Equ., 
{\bf 31}(1995), 1002--1012.

\bibitem{F07} A.V.~Faminskii, {\it  Nonlocal well-posedness of the mixed problem for the Zakharov--Kuznetsov equation}, J. Math. Sci., {\bf 147} (2007), 6524--6537.

\bibitem{FB07} A.V.~Faminskii and E.S.~Baykova, {\it Weak solutions to a mixed problem with two boundary conditions for generalized Zakharov--Kuznetsov equation}, in "Neklassicheskie Uravneniya Mat.Fiz." (ed.~A.I.Kozhanov), Novosibirsk (2007), 298--306 (in Russian).

\bibitem{FB08} A.V.~Faminskii and I.Yu.~Bashlykova, {\it Weak solution to one initial-boundary value problem with three boundary conditions for quasilinear equations of the third order}, Ukrainian Math. Bull., {\bf 5}(2008), 83--98.

\bibitem{F08} A.V.~Faminskii, {\it Well-posed initial-boundary value problems for the Zakharov--Kuznetsov equation}, Electronic J. Differential Equ., no.~127(2008), 1--23.

\bibitem{F12} A.V.~Faminskii, {\it Weak solutions to initial-boundary-value problems for quasilinear evolution equations of an odd order}, Adv. Differential Equ., {\bf 17}(2012), 421--470.

\bibitem{LSU} O.A.~Ladyzhenskaya, V.A.~Solonnikov and N.N.~Uraltseva, Linear and Quasilinear Equations of Parabolic Type, American Mathematical Society, Providence, Rhode Island, 1968.

\bibitem{LLS} D.~Lannes, F.~Linares and J.-C.~Saut, {\it The Cauchy problem for the Euler-Poisson system and derivation of the Zakharov--Kuznetsov equation}, Progress Nonlinear Differential Equ. Appl., {\bf 84}(2013), 183--215.

\bibitem{LT} N.A.~Larkin and E.~Tronco, {\it Regular solutions of the 2D Zakharov--Kuznetsov equation on a half-strip}, J. Differential Equ., {\bf 254}(2013), 81--101.

\bibitem{L} N.A.~Larkin, {\it Exponential decay of the $H^1$-norm for the $2D$ Zakharov--Kuznetsov equation on a half-strip}, J. Math. Anal. Appl., {\bf 405}(2013), 326--335.

\bibitem{L1} N.A.~Larkin, {\it The $2D$ Zakharov--Kuznetsov--Burgers equation on a strip}, 17 Apr. 2014, arXiv: 1404.4638v1 [math.AP].

\bibitem{LP} F.~Linares and A.~Pastor, {\it Well-posedness for the two-dimensional modified Zakharov--Kuznetsov equation}, SIAM J. Math. Anal., {\bf 41}(2009), 1323--1339.

\bibitem{LPS} F.~Linares, A.~Pastor and J.-C.~Saut, {\it Well-posedness for the Zakharov--Kuznetsov equation in a cylinder and on the background of a KdV soliton}, Comm. Partial Differential Equ., {\bf 35}(2010), 1674--1689.

\bibitem{S} J.-C.~Saut, {\it Sur quelques generalizations de l'equation de Korteweg--de~Vries}, J. Math. Pures Appl., {\bf 58}(1979), 21--61.

\bibitem{ST} J.-C.~Saut and R.~Temam, {\it An initial boundary value problem for the Zakharov--Kuznetsov equations}, Adv. Differential Equ., {\bf 15}(2010), 1001--1031.

\bibitem{STW} J.-C.~Saut, R.~Temam and C.~Wang, {\it An initial and boundary-value problem for the Zakharov--Kuznetsov equation in a bounded domain}, J. Math. Phys., {\bf 53}(2012), 115612, doi: 10.1063/1.4752102.

\bibitem{ZK} V.E.~Zakharov and E.A.~Kuznetsov, {\it On threedimentional solutions}, Zhurnal Eksp. Teoret. Fiz., {\bf 66}(1974), 594--597. English transl. in Soviet Phys. JETP, {\bf 39}(1974), 285--288.

\end{thebibliography}
\end{document}